\documentclass{amsart}

\usepackage{fullpage}

 \usepackage{amssymb,latexsym,amsmath,amsthm}

\usepackage[numbers,sort&compress]{natbib}

\usepackage{dsfont}

 \renewcommand{\div}{\mathop{\mathrm{div}}\nolimits}

\newtheorem*{thm*}{Theorem A}

\newtheorem{thm}{Theorem}[section]

\newtheorem{dfn}{Definition}[section]

\newtheorem*{note}{Notation}

\newtheorem{lemma}{Lemma}[section]

\newtheorem{prop}{Proposition}[section]

\newtheorem{cor}{Corollary}[section]

\newtheorem*{conj}{Conjecture}

\newtheorem*{conj*}{Conjecture A}

\newtheorem*{conj**}{Conjecture B}

 \numberwithin{equation}{section}

\begin{document}

\def\RR{{\mathbb{R}}}
\def\diver{\text{div}}
 \renewcommand{\div}{\mathop{\mathrm{div}}\nolimits}

\title[jump diffusion processes]{De Giorgi type results for equations with nonlocal lower-order terms}

\author{Mostafa Fazly}

\address{Department of Mathematics, The University of Texas at San Antonio, San Antonio, TX 78249, USA} \email{mostafa.fazly@utsa.edu}

\thanks{}

\maketitle

\begin{abstract}  
It is known that the De Giorgi's conjecture does not hold in two dimensions for semilinear elliptic equations with a nonzero drift,  in general, 
$$
\Delta u+ q\cdot \nabla u+f(u)=0 \ \ \text{in } \ \ \mathbb R^2, 
$$
when $q=(0,-c)$ for $c\neq 0$. This equation arises in the modeling of Bunsen burner flames.  Bunsen flames are usually made of two flames: a diffusion
flame and a premixed flame.  In this article, we prove De Giorgi type results, and stability conjecture, for the following local-nonlocal counterpart of the above equation (with a nonlocal premixed flame) in two dimensions, 
$$\Delta u + c L[u]  +  f(u)=0  \quad  \text{in} \ \  \RR^n, 
$$
when $L$ is a nonlocal operator, $f\in C^1(\mathbb R)$ and $c\in\mathbb R^+$. In addition, we provide a priori estimates for the above equation, when  $n\ge 1$,  with various jumping kernels. The operator $\Delta+cL$ is an infinitesimal generator of  jump-diffusion processes in the context of probability theory.  

\end{abstract}

\vspace{1cm}


\tableofcontents

\section{Introduction} 


Bonnet and Hamel in \cite{bh} studied  the existence of solutions of a reaction-diffusion equation in the plane $\mathbb R^2$.  The model is the following semilinear equation with an advection term 
\begin{equation}\label{deltac}
\Delta u - c \frac{\partial u}{\partial x_2} +f(u)=0 \ \ \text{in } \ \ \mathbb R^2,
\end{equation}
where $c$ is the speed constant and $f$ is a $C^1(\mathbb R)$ function. This problem arises in the modeling of Bunsen burner flames.  Bunsen flames are usually made of two flames: a diffusion
flame and a premixed flame. The authors in \cite{bh} constructed a solution $u$ for $c>0 $ and for the ignition type nonlinearity $f$ such that 
\begin{equation} 
u(\lambda k)\to -1 \ \ \text{when} \ \lambda \to \infty \ \text{for all} \ k=(\cos\phi,\sin\phi) \ \text{with} \ -\frac{\pi}{2}-\theta<\phi<-\frac{\pi}{2}+\theta,  
\end{equation}
and 
\begin{equation} 
u(\lambda k)\to 1 \ \ \text{when} \ \lambda \to \infty \ \text{for all} \ k=(\cos\phi,\sin\phi) \ \text{with} \ -\frac{\pi}{2}+\theta<\phi<\frac{3\pi}{2}-\theta,  
\end{equation}
for an angle $\theta\in(0,\frac{\pi}{2})$. This solution does not have one-dimensional symmetry due to the fact that level sets of solutions are parallel lines. In addition, the above solution $u$ is monotone in the direction of $x_2$-axis that is 
\begin{equation}
\frac{\partial u}{\partial x_2} >0  \ \ \text{in } \ \ \mathbb R^2. 
\end{equation}
This implies that the celebrated De Giorgi's conjecture does not hold for (\ref{deltac}) when $c\neq0$.   In other words, the De Giorgi's conjecture does not hold for semilinear elliptic equations with an advection term in two dimensions, that is 
\begin{equation}\label{delq}
\Delta u+ q\cdot \nabla u+f(u)=0 \ \ \text{in } \ \ \mathbb R^2, 
\end{equation}  
where $q$ is a vector and $q=(0,-c)$. We refer interested readers to  \cite{bhm}  by Berestycki, Hamel and Monneau and to \cite{faz} by the author for De Giorgi type results, called $m$-Liouville theorems, in this context.  In 1978, Ennio De Giorgi proposed a conjecture that reads;
\begin{conj}\label{conj1} 
Suppose that $u$ is an entire solution of the Allen-Cahn equation 
\begin{equation}\label{allen}
\Delta u+u-u^3 =0	 \quad \text{in}\ \  \mathbb {R}^n,
\end{equation}
satisfying $|u({x})| \le 1$, $\frac{\partial u}{\partial x_n} ({x}) > 0$ for ${x} = ({x}',x_n) \in \mathbb{R}^n$.	
Then, at least in dimensions $N\le 8$ the level sets of $u$ must be hyperplanes, i.e. there exists $g \in C^2(\mathbb{R})$ such that $u({x}) = g(\tau\cdot {x} - c)$, for some fixed $\tau \in \mathbb{R}^{n}$ when $\tau_n>0$.
\end{conj}
If monotonicity is replaced by stability, this is known as the {\it stability conjecture}. The De Giorgi's conjecture was established by Ghoussoub and Gui in \cite{gg1} 
in two dimensions. In fact, the proof is valid for the stability conjecture and for any $f\in C^1(\mathbb R)$ that is 
\begin{equation}\label{allen}
\Delta u+f(u) =0	 \quad \text{in}\ \  \mathbb {R}^n,
\end{equation}
as it is structured based on a  linear Liouville-type theorem for elliptic equations in divergence form, see (\cite{bbg, gg1}). Ambrosio and Cabr\'{e} in \cite{ac}, and later with  Alberti in \cite{aac}, extended the result to dimension $n= 3$ by adjusting the linear Liouville theorem. 
Ghoussoub and Gui also showed in \cite{gg2} that the conjecture holds for $n = 4$ or $n = 5$ for solutions that satisfy certain antisymmetry conditions, and Savin in \cite{savin} established its validity   for $4 \le n \le 8$ under the following additional natural hypothesis on the solution,
 \begin{equation}\label{asymp}
\lim_{x_n\to\pm\infty } u({x}',x_n)\to \pm 1  \quad \text{for}\ \  x'\in\mathbb {R}^{n-1}. 
\end{equation}
In dimension $n \ge 9$,  del Pino-Kowalczyk-Wei in \cite{dkw} gave a counterexample to De Giorgi's conjecture which long believed to exist. 
 Under a much stronger assumption that the limits in \eqref{asymp} are uniform in ${x}'$, the conjecture is known as  {\it Gibbons' conjecture}. This conjecture was proved for all dimensions independently  with different methods  by Barlow, Bass and Gui in \cite{bbg}, Berestycki, Hamel and Monneau in \cite{bhm} and Farina in \cite{far}.

 In this article, we consider a nonlocal counterpart of (\ref{delq})  where the Bunsen fames are  made of two fames: a diffusion
fame and a nonlocal premixed fame, 
  \begin{equation} \label{main}
\Delta u + c L[u]  +  f(u)=0  \quad  \text{in} \ \  \RR^n , 
  \end{equation}   
  when $c$ is a positive  constant and the operator $L$ is defined by the nonlocal operator 
 \begin{equation} \label{L}
L[u(x)] :=   \lim_{\epsilon\to 0} \int_{\{y\in \mathbb R^n, |x-y|>\epsilon\} } [u(y) - u(x)] J (x,y) dy. 
    \end{equation}   
   We suppose that $f\in C^1(\mathbb R)$  and   $J$ is a  nonnegative measurable symmetric  even jump kernel, unless otherwise is stated. We establish De Giorgi type results  for bounded stable solutions and various energy estimates for this equation. In probability theory, such operators have been studied extensively and they are known as jump-diffusion processes and Brownian motions with Gaussian components,  see \cite{bku,cg, cks}.  A (rotationally) symmetric $\alpha$-stable process $Y=(Y_t, t\ge 0, \mathbb P_x, x\in\mathbb R^n)$ in $\mathbb R^n$ is a L\'{e}vy process that 
\begin{equation}
\mathbb E_x\left[ e^{i\zeta\cdot (Y_t-Y_0)}\right]= e^{-t|\zeta|^\alpha} \ \ \text{for every} \ x,\zeta\in\mathbb R^n. 
\end{equation}
The infinitesimal generator of a symmetric $\alpha$-stable process $Y$ in $\mathbb R^n$ is the fractional Laplacian operator $\Delta^{\frac{\alpha}{2}}$ that is  a prototype of nonlocal operators when $0<\alpha<2$. The fractional Laplacian operator is of the form
\begin{equation}
\Delta^{\frac{\alpha}{2}} u(x) :=  \lim_{\epsilon\to 0} \int_{\{y\in \mathbb R^n, |x-y|>\epsilon\} } [u(y) - u(x)] \frac{\mathcal C (n,\alpha)}{|x-y|^{n+\alpha}} dy, 
\end{equation}
for $\mathcal C(n,\alpha):=\alpha 2 ^{\alpha-1}\pi^{-n/2} \Gamma\left(\frac{n+\alpha}{2}\right)\Gamma\left(1-\frac{\alpha}{2}\right)^{-1}$,  and the jumping kernel in (\ref{L}) is
\begin{equation}\label{JFrac}
  J(x,y)=  \mathcal C(n,\alpha)|x-y|^{-n-\alpha}.
  \end{equation} 
The operator $L=\Delta^{\frac{\alpha}{2}}$ when $0<\alpha<2$ can be regarded as an interpolation of identity and Laplacian in the sense that 
\begin{equation}
\lim_{\alpha\to 0^+ } \Delta^{\frac{\alpha}{2}} u = u \ \ \text{and} \ \ \lim_{\alpha\to 2^- }  \Delta^{\frac{\alpha}{2}} u =  \Delta u, 
\end{equation}
under certain conditions, see \cite{dpv}.  Note that the kernel  (\ref{JFrac}) is a particular case of the following kernel, known as ellipticity condition for operator $L$, 
\begin{equation}\label{Jumpc}
 J (x,z) =  c(x-z) |x-z|^{-n-\alpha}, 
 \end{equation} 
where $c(x-z)$ is bounded between two positive constants $0<\lambda \le \Lambda$. 
 It is by now a well-known, see \cite{cas} by Caffarelli and Silvestre,  that the fractional Laplacian operator can be realized as the boundary operator (more precisely the Dirichlet-to-Neumann operator) of a suitable extension function in the half-space.  Now, let $X^0$ be a Brownian motion in $\mathbb R^n$ with generator $\Delta $ Laplacian operator and $Y$ be a symmetric $\alpha$-stable process in $\mathbb R^n$. Assume that $X^0$ and $Y$ are independent. Consider the process $X^a$ given by 
\begin{equation}
X^a_t:=X^0_t+aY_t,
\end{equation}
 that is the independent sum of the Brownian motion  $X^0$ and  the
symmetric $\alpha$-stable process $Y$ with weight $a>0$.   The infinitesimal generator of $X^a$
 is the elliptic operator 
 \begin{equation}
 M_{a,\alpha}:=\Delta +c \Delta ^{\frac{\alpha}{2}} \ \ \ \text{for} \ \ c:=a^\alpha,   
 \end{equation}
  and the function $J_a(x,y)=a^\alpha  J(x,y)$, when $J$ is given by (\ref{JFrac}),  is the L\'{e}vy  intensity of $X^a$. Various aspects of the process $X^a$ and the operator $M_{a,\alpha}$, such as boundary Harnack principle (BHP), De Giorgi-Nash-Moser-Aronson type theory and  heat kernel and Green's function estimates,  are studied in the literature. In this regard we refer interested readers to series of article by Chen et  al.  in \cite{cksv1, cksv2, cksv3, cksv4, ck} and references therein. Let ${\hat X}^a$ be a L\'{e}vy process obtained from $X^a$ by eliminating all its jumps of larger than $\delta_0$. Then, the infinitesimal generator of $\hat X$ is 
   \begin{equation}
 \hat M_{a,\alpha}:=\Delta +c {\hat \Delta}^\frac{\alpha}{2},   
 \end{equation}
 where 
  \begin{equation}\label{hatfrac}
{\hat\Delta}^{\frac{\alpha}{2}} u(x) :=  \lim_{\epsilon\to 0} \int_{\{y\in \mathbb R^n, \epsilon<|x-y|<\delta_0 \} } [u(y) - u(x)] \frac{\mathcal C(n,\alpha)}{|x-y|^{n+\alpha}} dy. 
\end{equation}
Note that ${\hat\Delta}^{\frac{\alpha}{2}}$ is associated to nonlocal operator $L$ in (\ref{L}) with the truncated jump kernel 
 \begin{equation}
  J(x,y)=  \mathcal C(n,\alpha)|x-y|^{-n-\alpha} \mathds{1}_{\{|x-y|\le \delta_0\}}. 
  \end{equation}

It is also known in the probability theory that suitable
 estimates for the L\'{e}vy process $X^a$ can be obtained from $\hat X^a$  by adding back the jumps of $X^a$ of size larger than $\delta_0$, see \cite{cr} where Schramm-L\"{o}wner
evolutions are studied in the light of one-dimensional symmetric stable processes. Stochastic processes with truncated jump kernels, known also as finite range jump processes, and their associated  infinitesimal generators are studied in the literature, in regards to probability theory see \cite{bku, ckk, bbck, cksv3, fot} and in regards to elliptic partial differential equations see \cite{cp,fgu,fs,hrsv, bbg} and 
references therein.   In addition to above jumping kernels,  the following truncated kernels, which are locally comparable to (\ref{JFrac}) are of our interests 
 \begin{equation}\label{Jumpki}
\frac{\lambda}{|x-y|^{n+\beta}} \mathds{1}_{\{|x-y|\le \delta_1\}} \le  J  (x-y) \le  \frac{\Lambda}{|x-y|^{n+ \alpha}} \mathds{1}_{\{|x-y|\le \delta_0\}}, 
 \end{equation} 
 when  $0<  \delta_1 \le  \delta_0$, $0<\lambda\le \Lambda$ and $0<  \beta\le  \alpha<2$. In the context of classical De Giorgi's conjecture and in order to establish Gibbons' conjecture and to establish a linear Liouville theorem, Barlow, Bass and Gui in \cite{bbg} studied generators and Dirichlet forms of  symmetric processes with truncated jump kernels of the form 
 \begin{equation}
 J_0(x,y)= |x-y|^{-(n+1)} \mathds{1}_{\{|x-y|\le 1\}} \ \ \text{ for all } x,y\in\mathbb R^n, \ \ x\neq y. 
 \end{equation}
 Note that this is a particular case of (\ref{Jumpki}) that represents ${\hat \Delta}^{\frac{1}{2}}$, as given in (\ref{hatfrac}). In addition, they considered kernels with decays of the form 
 \begin{eqnarray}\label{}
&&\frac{c_1}{|x-y|^{n+1}} \le  J_1  (x,y) \le  \frac{c_2}{|x-y|^{n+ 1}} , \ \ \text{for} \ \  |x-y|\le 1 \ \ \text{and}
\\&& 
  \int_{|x-y|>r} J_1   (x,y)  dy \le c_0 e^{-\alpha_0 r}\ \ \ \text{when} \ \ \ r > 1 , 
  \end{eqnarray}
  where $c_i$ for $i=0,1,2$ and $\alpha_0$ are positive constant. In addition, Chen et al.  in \cite{cksv4,ckk2} considered rotationally symmetric L\'{e}vy processes on $\mathbb R^n$ whose L\'{e}vy measure decays exponentially near
infinity at exponential rate $e^{-r^\beta}$ with $\beta>1$.   Inspired by the above,  we study generators of L\'{e}vy processes  obtained from $X^a$ by eliminating all its jumps of larger than $\delta_0$ and replacing those with jumps with certain decay rates. More precisely, we consider 
\begin{eqnarray}\label{JDecay}
&&\frac{\lambda}{|x-y|^{n+\beta}} \le J (x-y) \le  \frac{\Lambda}{|x-y|^{n+\alpha }} \ \ \ \text{when} \ \ \ |x-y|\le \delta_0  \ \ \text{and}
\\
&&\label{JDecayr}
\int_{r<|x-y|<2r} J   (x-y)  dy \le C D (r)\ \ \ \text{when} \ \ \ r > \delta_0 ,  
 \end{eqnarray} 
for an appropriate algebraic decay function  $0\le D \in C(\mathbb R^+)$ with $\lim_{r\to\infty} D (r)=0$. We shall fix function $D$ later. The quadratic form $(\mathcal I, \mathcal F)$, also called the Dirichlet form, associated with the generator $-\Delta - c L$ is given by $\mathcal F:=W^{1,2}(\mathbb R^n)$ and for $u,v\in\mathcal F$,  
\begin{equation}\label{mathcalI} 
\mathcal I(u,v) :=\int_{\mathbb R^n} \nabla u(x)\cdot \nabla v(x) dx+\frac{c}{2} \iint_{\RR^{2n}} [u(x)-u(y)][v(x)-v(y)] J(x,y) dx dy.  
\end{equation}
 The associated  energy functional for solutions of  (\ref{main}) on 
 $\Omega\subset \mathbb R^n$ is 
\begin{equation}\label{energy}
\mathcal E(u,\Omega):=\mathcal E_J^{\text{Sob}}{(u,\Omega)} + \mathcal E^{\text{Pot}} {(u,\Omega)} , 
\end{equation}
and  functionals  $\mathcal E^{\text{Sob}}$ and $\mathcal E^{\text{Pot}} $ are given by 
\begin{equation}\label{Esob}
\mathcal E_J^{\text{Sob}}(u,\Omega):=  \frac{1}{2} \int_{\Omega} |\nabla u(x)|^2 dx + \frac{c}{2} \iint_{\mathbb R^{n}\times \mathbb R^{n} \setminus  \mathcal C\Omega \times   \mathcal C\Omega}  | u(x) -u(y) |^2 J (x-y) dy dx, 
\end{equation}
when $\mathcal C\Omega = \mathbb R^n\setminus \Omega$ and 
\begin{equation}\label{Epot}
\mathcal E^{\text{Pot}} {(u,\Omega)} :=\int_{\Omega}  F(u(x)) dx, 
\end{equation}
when $F\in C^1(\mathbb R)$  is an antiderivative of $-f$.

\begin{dfn} A solution $u$ of (\ref{main}) is called stable if the second variation of $\mathcal E$ at $u$ is
nonnegative, that is for any $\zeta\in C_c^1(\mathbb R^n)$, 
\begin{equation} \label{stability}
\int_{\RR^n}  f'(u(x)) \zeta^2(x) dx \le \frac{1}{2} \int_{\mathbb R^n} |\nabla \zeta(x)|^2 dx+ \frac{c}{2} \iint_{\RR^{2n}}   [\zeta(x)- \zeta(y)]^2 J(x-y) dy dx . 
\end{equation} 
\end{dfn}

\begin{dfn}\label{dfnsym} 
We call $\Gamma_{2R,R}=\cup_{i=1}^6\Gamma^i_R$ a symmetric domain decomposition of $\mathbb R^n\times \mathbb R^n$ when every $\Gamma^i_R$ is given by
\begin{eqnarray*}\label{}
&&  \Gamma^1_R:=B_{ R}\times (B_ {2R}\setminus B_{R}),  \Gamma^2_R:=(B_ {2R}\setminus B_{ R})\times (B_ {2R}\setminus B_{ R}), \Gamma^3_R:= (\mathbb R^n\setminus B_{2R}) \times (B_ {2R}\setminus B_{ R}) ,\ 
\\&& \label{gamma3}   \ \Gamma^4_R:=B_{ R}\times (\mathbb R^n\setminus B_{2R}), \ \Gamma^5_R:=B_{R}\times B_{R},  \ \Gamma^6_R:=(\mathbb R^n\setminus B_{2R})\times (\mathbb R^n\setminus B_{2R}). 
\end{eqnarray*}
\end{dfn}

The structure of the article as it follows. In Section \ref{secmain}, we provide our main results. In Section \ref{secthm1}, we prove a Poincar\'{e} type inequality, a Liouville theorem and  De Giorgi type results. In Section \ref{secen}, we prove energy estimates for jump-diffusion processes with various  jump kernels.  In Section \ref{secham}, we prove a Modica type pointwise estimate and a Hamiltonian identity in one dimension,  and monotonicity formulae in higher dimensions.   Section \ref{secsum} is devoted to summation of nonlocal operators.

\section{Main Results; Statements}\label{secmain}
In this section we present our main results of this article. We start with a linear Liouville theorem for a local-nonlocal operator. This theorem is inspired by a classical one for the Laplacian operator that was noted by Berestycki, Caffarelli and Nirenberg in \cite{bcn} and by Barlow, Bass and Gui in \cite{bbg}, see also \cite{bar}, used by Ghoussoub and Gui \cite{gg1} and later by Ambrosio and Cabr\'{e} \cite{ac}  to prove the De Giorgi conjecture in dimensions two and three. 
For the case of fractional Laplacian, using the Caffarelli and Silvestre extension function in \cite{cas}, this type linear Liouville theorem is given by Cabr\'{e} and  Sol\'{a}-Morales in \cite{csol} and by Cabr\'{e} and Sire  in \cite{cabSire1,cabSire2}.  For more general nonlocal operators, we refer to Hamel et al. in \cite{hrsv} and to Sire and the author in \cite{fs}.

\begin{thm}\label{thmlione} Let $\phi \in L^{\infty}_{loc}(\mathbb{R}^n) $ and $\sigma \in H^1_{loc}(\mathbb{R}^n)$  such that $\phi^2 >0 $ a.e.,  and they satisfy  
\begin{equation}\label{linPhi}
\div(\phi^2(x)\nabla \sigma(x)) +    \lim_{\epsilon\to 0} \int_{\{y\in \mathbb R^n, |x-y|>\epsilon\} }  \left( \sigma(y)- \sigma(x) \right)\phi(x) \phi(y)  J(x-y) dy  = 0 \ \ \text{in} \ \ \mathbb R^n , 
\end{equation}
  when $J$ is a nonnegative  measurable symmetric even jumping kernel.  Assume also that $\Gamma_R$ is a symmetric domain decomposition, in Definition \ref{dfnsym}, and for $R>1$
\begin{equation}\label{Phiuxuy}
\int_{\RR^{n}\cap (B_{2R}\setminus B_R)}  \phi^2\sigma^2 dx +   \iint_{\RR^{2n}\cap \{\cup_{k=1}^4 \Gamma^k_R \}}   [\sigma(x) +  \sigma(y)]^2  \phi(x) \phi(y) |x-y|^2 J(x-y)   dy dx \le C R^2 ,
 \end{equation}
where $\Gamma_{2R,R}=\cup_{k=1}^6\Gamma^k_R$ is a symmetric domain decomposition.  Then,  $\sigma$ must be constant. 
\end{thm}

The next theorem is a Poincar\'{e} type inequality for stable solutions of (\ref{main}). For the case of local semilinear equations, this inequality was established by Sternberg and Zumbrun in \cite{sz1,sz2}. The inequality was used in \cite{fsv} to prove certain De Giorgi type results for scalar equations and in \cite{fgo} for multi-component systems.  For the case of nonlocal equations, such an inequality was derived in  \cite{fs,fgu,cf,sv} and references therein. 

\begin{thm}\label{ThmPoin}
 Assume that  $n\ge 1$ and $ u$ is a stable solution of (\ref{main}).  Then,  for any $\eta \in C_c^1(\mathbb R^{n})$, 
\begin{eqnarray}\label{poinsysm1}
&&\int_{  \mathbb R^{n}\cap  \{|\nabla_x u|\neq 0\}}   \left(   |\nabla u|^2 \mathcal{\kappa}^2 + | \nabla_T |\nabla u| |^2  \right)\eta^2 dx +\frac{c}{2}\iint_{  \mathbb R^{2n}\cap  \{|\nabla_x u|\neq 0\}}    \mathcal A_y(\nabla_x u)  [\eta^2(x)+\eta^2(x+y)] J(y) dx dy 
\\&\le& \nonumber  \int_{\mathbb R^n} |\nabla u|^2 |\nabla \eta|^2 dx+ \frac{c}{2}  \iint_{  \mathbb R^{2n}}  \mathcal B_y(\nabla_x u) [ \eta(x) - \eta(x+y) ] ^2 J(y) d x dy  , 
  \end{eqnarray} 
where $\nabla_T$ stands for the tangential gradient along a given level set of $u$ and 
$\mathcal{\kappa}^2$ for the sum of the squares of the principal curvatures of such a level set, and 
\begin{eqnarray}\label{mathcalA}
\mathcal A_y(\nabla_x u) &:= & |\nabla_x u(x)|  |\nabla_x u(x+y)| -\nabla_x u(x)  \cdot \nabla_x u(x+y) , 
\\ \label{mathcalB}
 \mathcal B_y(\nabla_x u)  & :=& |\nabla_x u(x)| | \nabla_x u(x+y)| .
 \end{eqnarray}
 \end{thm}
We now provide De Giorgi type results for  stable solutions of  (\ref{main}) for truncated or finite range jump kernels in two dimensions. This settles the stability conjecture for jump-diffusion operators. In order to establish this theorem, we apply the Poincar\'{e} type inequality (\ref{ThmPoin}). In addition, the linear Liouville theorem, that is Theorem \ref{thmlione}, can be applied to prove the following as well.   

\begin{thm}\label{Thmain} Let $u$  be a bounded stable solution of  (\ref{main}) in two dimensions when the jump kernel $J$ is truncated and satisfies  (\ref{Jumpki}) for $0<\alpha,\beta<2$.  Then,  $u$ must be a one-dimensional function.   
  \end{thm}
 The next result is the counterpart of the above result for the case of jump kernels with an algebraic decay rates at infinity. 
  
  \begin{thm}\label{Thmain2} Let $u$  be a bounded stable solution of  (\ref{main}) in two dimensions when the jump kernel $J$ satisfies (\ref{JDecay}) for $0<\alpha,\beta<2$ with a decay rate (\ref{JDecayr}) when $D(r)<C r^{-\theta}$ for $\theta>3$.  Then,  $u$ must be a one-dimensional function.   
  \end{thm}

We now provide energy estimates for monotone  solution of (\ref{main}) when the kernel is either truncated or with decays at infinity. The remarkable point is that the upper bound for the energy is $R^{n-1}$, for any $0<\alpha<2$,  that is the same as the one for the Laplacian operator. For the Laplacian operator, this energy bound was established by Ambrosio and Cabr\'{e} in \cite{ac}.  

\begin{thm}\label{thmene}
Let $u$ be a bounded monotone  solution of (\ref{main}) satisfying (\ref{asymp}) and  $F(\pm1)=0$.   Assume also that the kernel $J$ is truncated and satisfies (\ref{Jumpki}) or  the kernel $J$ has decays as in (\ref{JDecay})-(\ref{JDecayr}) with decay-rate  $D(r)< C r^{-\theta}$ when $\theta>2$ for all $r>\delta_0$.   
Then,     
\begin{equation}\label{EKRnminus1}
\mathcal E(u,B_R) \le C  R^{n-1} \ \ \text{for} \ \  R>\delta_0, 
\end{equation}
where the positive constant $C$ is independent from $R$ but may depend on $\delta_0,\alpha,n$.  
\end{thm}
In the following theorem, we provide a counterpart of the above estimate when the jump kernel $J$ is the fractional Laplacian or,  in the more general context,  it satisfies the ellipticity condition. The following energy estimate is consistent with the fractional Laplacian operator for any $0<\alpha<2$.  In this regard, we  refer interested readers to \cite{cc,cc2,fs,fgu,cp,psv} and references therein. 

\begin{thm}\label{thmeneal}
Suppose that $ u$ is a bounded monotone solution of (\ref{main}) with  $ F(\pm 1)=0$.  Assume also that the kernel $J$ satisfies  (\ref{Jumpc}).   Then,  the following energy estimates hold for $R>1$. 
\begin{enumerate}
\item[(i)] If $0<\alpha<1$, then $\mathcal E(u,B_R)  \le  C R^{n-\alpha}$,
\item[(ii)] If  $\alpha=1$, then $\mathcal E(u,B_R) \le  C R^{n-1}\log R$,
\item[(iii)] If $\alpha>1$, then $\mathcal E(u,B_R)  \le  C R^{n-1}$,
\end{enumerate} 
where the positive constant $C$ is independent from $R$ but may depend on $\alpha, n,\lambda,\Lambda$. 
\end{thm}

It is straightforward to generalize our main results in this section to the case when the jumping measure is of the form 
\begin{equation}
J(x,y) = \frac{1}{|x-y|^n \phi(|x-y|)}. 
\end{equation} 
Here,  $\phi$ is a positive increasing function in $\mathbb R^+$ satisfying certain conditions, see \cite{ck,cksv4} and references therein in the context of stochastic processes. 

One of the main difficulties in proving properties for operator $-T=\Delta + c L$ is that it shows 
different scales for the local and nonlocal parts. The diffusion part has Brownian scaling while the
jump part has a different type of scaling depending on the kernel. For the rest of this section,  we provide statements of an Hamiltonian identity and a Modica-type estimate  in one dimension and a monotonicity formula  for radial solutions for the case of $L=\Delta^{\frac{\alpha}{2}}$. Consider the fractional Laplacian operator 
\begin{equation}\label{g}
 (- \Delta)^{\frac{\alpha}{2}} u = g(x) \ \ \text{in}\ \ \mathbb{R}^n, 
\end{equation} 
when $g\in C^{2,\gamma}(\mathbb{R}^n)$ for $\gamma > \max(0,1-\alpha)$.  As it was mentioned in previous sections,  the fractional laplacian can be realized as the boundary operator (more precisely the Dirichlet-to-Neumann operator) of a suitable extension in the half-space, see \cite{cas}. In the light of this,  Caffarelli and Silvestre introduced an extension function $v(x,y)$  of the solution  $u(x)$ of (\ref{g}) that  satisfies 
 \begin{eqnarray}\label{eg}
 \left\{ \begin{array}{lcl}
\hfill \div(y^{a} \nabla v)&=& 0   \ \ \text{in}\ \ \mathbb{R}_+^{n+1}=\left \{x \in \mathbb R^n, y>0 \right \},\\   
\hfill -\lim_{y\to0}y^{a} \partial_{y} v &=& d_{\alpha} g(x)   \ \ \text{in}\ \ \partial\mathbb{R}_+^{n+1},
\end{array}\right.
  \end{eqnarray}
when $a=1-\alpha$ and $d_\alpha$ is a constant.  We use the above extension problem to establish  a Hamiltonian identity in one dimension, following ideas and methods established by Cabr\'e and  Sire
 in \cite{cabSire1,cabSire2} for the fractional Laplacian operator.  

\begin{thm}\label{hamilton}
Let  $n=1$ and $v(x,y)$ be the extension function of solution $u(x)$ of \eqref{main} satisfying 
\begin{equation}\label{asympve}
\lim_{x\to\infty } v({x}, 0)\to  \tau \in \mathbb R .
\end{equation}
 Then, the following identity holds for any $x\in \mathbb R$
\begin{equation}\label{hamiltonian}
d_\alpha  (\partial_x v(x,0))^2+c \int_0^\infty y^{a} \left[  (\partial_x v)^2 - (\partial_y v)^2 \right] dy =2d_\alpha \left[  F( v(x,0)) - F( \tau) \right]. 
\end{equation}
\end{thm}

Note that if $\lim_{x\to\pm\infty } v({x}, 0)\to  \tau^{\pm}$  where $\tau^{\pm}$ are constant, 
then $F(\tau^+)=F(\tau^-)$. We now provide a Modica-type pointwise estimate in one dimension. 

\begin{thm}\label{modica}
Let  $n=1$ and $v(x,y)$ be the extension function of solution $u(x)$ of \eqref{main} satisfying monotonicity condition $v_x(x,0)>0$ and (\ref{asympve}).  Then, the following pointwise inequality holds  for any $(x,y)\in \mathbb R\times  \mathbb R^+$
\begin{equation}\label{modicaest}
d_\alpha  (\partial_x v(x,0))^2+c \int_0^y  t^{a} \left[  (\partial_x v(x,t))^2 - (\partial_y v(x,t))^2 \right] dt <2d_\alpha \left[  F( v(x,0)) - F( \tau) \right]. 
\end{equation}
\end{thm}
 Note that for pure local problem $\Delta u=f(u)$ in $\mathbb R^n$, Modica  in \cite{mod} established the celebrated inequality $|\nabla u|^2 \le 2 F(u)$ in $\mathbb R^n$ when $u$ is bounded and $F\ge 0$. Naturally, one may ask if this result for the local problem could be used to prove a counterpart of (\ref{modicaest}) in higher dimensions $n\ge 2$. This remains as an open problem. We end this section with the following monotonicity formula for radial solutions. 

\begin{thm}\label{iden}
Let $n\ge 1$ and $ v=v(|x|,y)$ be a bounded radial extension function for solution $u(x)$ of \eqref{main}. Then, the following function $I(r)$ is nonincreasing in terms of $r$,
  \begin{equation}\label{radialmono}
I(r):=d_\alpha  (\partial_r v(r,0))^2+c \int_0^\infty y^{a} \left[  (\partial_r v)^2 - (\partial_y v)^2 \right] dy - 2d_\alpha  F( v(r,0) ) \ \ \text{for} \ \ r\in\mathbb R^+. 
  \end{equation}
\end{thm}
For nonradial solutions, a (weak) monotonicity formula is provided in Section \ref{secham}. Pointwise estimates and monotonicity formulae are fundamental tools in proving rigidity results in this context.

\section{De Giorgi type Results; Proofs of Theorem \ref{thmlione}-\ref{Thmain2}}\label{secthm1}

We start this section with showing that monotone solutions are stable solutions that is inequality (\ref{stability}) holds. We refer to this inequality as stability inequality.  
\begin{prop}\label{Propstab}  
  Let $u$ be a monotone solution of (\ref{main}).  Then,  for any $\zeta\in C_c^1(\mathbb R^n)$, 
\begin{equation} \label{stability2}
\int_{\RR^n}  f'(u(x)) \zeta^2(x) dx \le \int_{\mathbb R^n} |\nabla \zeta(x)|^2 dx+ \frac{c}{2} \iint_{\RR^{2n}}   [\zeta(x)- \zeta(y)]^2 J(x-y) dx dy . 
\end{equation} 
\end{prop}  
In order to prove the above inequality, we first provide a technical lemma. We omit its proof since it is elementary. 
\begin{lemma}\label{lemtech}
Let $T=-\Delta -cL$ and  the jump kernel  $J$ be  measurable, symmetric and even.   Suppose  that $f,g\in C^1(\mathbb R^n)$. Then, 
\begin{equation}
\int_{\mathbb R^n} g(x)T(f(x)) dx=\int_{\mathbb R^{n}} \nabla f(x)\cdot \nabla g(x) dx+ \frac{c}{2}\iint_{\RR^{2n}}    [f(x) - f(y)]  \left[g(x)-g(y)  \right] J(x-y) dx dy, 
\end{equation}
where the right-hand side of the above is $ \mathcal I(f,g) $. 
\end{lemma}

\vspace*{.2 cm}

\noindent {\it Proof of Proposition \ref{Propstab}}.  Let $u$ denote a monotone solution of (\ref{main}). Then,  differentiating with respect to $x_n$ and calling $\phi:=\frac{\partial u}{\partial x_n}>0$ we have  
 \begin{equation} \label{L12}
T[\phi]= f'(u) \phi   \ \ \ \text{in}\ \ \mathbb R^n. 
  \end{equation} 
Now, multiply both sides with $\frac{\zeta^2}{\phi}$ for $\zeta\in C_c^1(\mathbb R^n)$ to get  
\begin{equation*} \label{}
T[\phi] \frac{\zeta^2}{\phi}=  f'(u) \zeta^2   \ \ \ \text{in}\ \ \mathbb R^n.
  \end{equation*} 
 From this and (\ref{L12}) we get 
\begin{equation} \label{LL1}
 \int_{\RR^n}  f'(u(x)) \zeta^2(x)  dx  \le  \int_{\RR^n} T [\phi(x)] \frac{\zeta^2(x)}{\phi(x)} dx . 
  \end{equation} 
Applying Lemma \ref{lemtech} for the right-hand side of the above, we have
\begin{eqnarray}\label{Lphi}
\int_{\RR^n} T[ \phi(x)] \frac{\zeta^2(x)}{\phi(x)} dx &=&  \int_{\RR^n} \nabla \phi(x) \cdot \nabla \frac{\zeta^2(x)}{\phi(x)} dx\\&&+ \nonumber
\frac{1}{2} \iint_{\RR^{2n}}    [\phi(x) - \phi(y)] \left[ \frac{\zeta^2(x)}{\phi(x)}-  \frac{\zeta^2(y)}{\phi(y)} \right] J (x-y) dx dy. 
\end{eqnarray}
Note that for $a,b,c,d\in\mathbb R$ when $ab<0$ we have 
\begin{equation*}
(a+b)\left[  \frac{c^2}{a} +  \frac{d^2}{b}  \right] \le (c-d)^2    .
\end{equation*}
Since each $\phi$ is positive,  we have $\phi(x)\phi(z)>0$. Setting  $a=\phi(x)$, $b=-\phi(y)$, $c=\zeta(x)$ and  $d=\zeta(y)$ in the above inequality and from the fact that $ab=-\phi(x)\phi(y)<0$, we  conclude 
\begin{equation}\label{phixy}
[\phi(x) - \phi(y)] \left[ \frac{\zeta^2(x)}{\phi(x)}-  \frac{\zeta^2(y)}{\phi(y)} \right] \le [\zeta(x)- \zeta(y)]^2    .
\end{equation} 
 On the other hand, for $a,b\in\mathbb R^n$ we have $ 2a\cdot b - |b|^2 \le |a|^2.$ Setting $a=\nabla \zeta$ and $b=\nabla \phi\frac{\zeta}{\phi}$ in the latter inequality,  we get
 \begin{equation}\label{zetaxy}
 2\nabla \zeta\cdot\nabla \phi \frac{\zeta}{\phi} - |\nabla \phi|^2 \frac{\zeta^2}{\phi^2}\le |\nabla \zeta|^2.
 \end{equation} 
  From (\ref{Lphi}), (\ref{phixy}) and (\ref{zetaxy}) we conclude  
\begin{equation*}
\int_{\RR^n} T[\phi(x)] \frac{\zeta^2(x)}{\phi(x)} dx\le   \int_{\mathbb R^n} |\nabla \zeta(x)|^2 dx +  \frac{1}{2} \int_{\RR^n} \int_{\RR^n}   [\zeta(x)- \zeta(y)]^2 J(y-x) dy dx.\end{equation*} 
This  and (\ref{LL1}) complete the proof.

\hfill $\Box$

We now provide proofs for our main results. 

\vspace*{.2 cm}

\noindent {\it Proof of Theorem \ref{ThmPoin}}. Let  $u$ be a stable solution of (\ref{main}). Set $\zeta(x)=|\nabla _x u(x)| \eta(x)$ in the stability inequality for $\eta \in C_c^1(\mathbb R^n)$,  
 \begin{eqnarray*} \label{Testab}
 \int_{\RR^n} f'(u) |\nabla_x u (x)|^2   \eta^2(x)  dx &\le& \int_{\mathbb R^n} |\nabla (|\nabla _x u| \eta)  |^2 dx\\&&+  \frac{c}{2}  \iint_{\RR^{2n} }  [ |\nabla_x u(x)|  \eta(x)-  |\nabla_x u(x+y)| \eta(x+y)]^2 J (y) dy dx =:I_1+I_2. 
\end{eqnarray*} 
We now compute $I_1$ and $I_2$ in the right-hand side of the above inequality
 \begin{equation*} \label{}
I_1=   \int_{\mathbb R^n} |\nabla |\nabla u| |^2 {\eta^2}+  \int_{\mathbb R^n} |\nabla \eta|^2|\nabla u|^2 +\frac{1}{2}  \int_{\mathbb{R}^n} \nabla |\nabla u|^2\cdot\nabla {\eta^2} , 
\end{equation*} 
and 
 \begin{eqnarray*} \label{}
I_2&=&  \frac{c}{2} \iint_{\RR^{2n}}  |\nabla_x u(x)|^2  \eta^2(x) J(y) dy dx + \frac{c}{2} \iint_{\RR^{2n}}   |\nabla_x u(x+y)|^2  \eta^2(x+y) J(y) dy dx
\\&& - c \iint_{\RR^{2n}}     |\nabla_x u(x)| |\nabla_x u(x+y)| \eta(x) \eta(x+y) J(y) dy dx. 
\end{eqnarray*} 
We now apply the equation (\ref{main}). Note that for any index $1\le k\le n$ we have 
\begin{eqnarray*}
T [ \partial_{x_k} u(x)] = -\Delta  \partial_{x_k} u(x) + c
\int_{\RR^n } [ \partial_{x_k} u(x) - \partial_{x_k} u(x+y)]   J(y) dy 
= f'(u) \partial_{x_k} u(x) . 
\end{eqnarray*}
Multiplying both sides of the above equation with $\partial_{x_k} u(x) \eta^2(x)$  and integrating we have
\begin{equation*}
 \int_{\mathbb R^{n}} f'(u) [\partial_{x_k} u(x)]^2  \eta^2(x)
dx =  - \int_{\mathbb R^n} \Delta  \partial_{x_k} u(x) [\partial_{x_k} u(x) \eta^2(x)] dx  +c \int_{\mathbb R^n} \partial_{x_k} u(x) \eta^2(x) L [\partial_{x_k} u(x)   ] dx . 
\end{equation*}
We now simplify the right-hand side of the above, using Lemma \ref{lemtech}, as  
\begin{eqnarray*}
&&\int_{\mathbb R^{n}} |\nabla \partial_{x_k} u|^2\eta^2 dx + \frac{1}{2} \int_{\mathbb R^{n}} \nabla |\partial_{x_k} u|^2 \cdot\nabla \eta^2  dx 
\\&&+\frac{c}{2}  \iint_{\mathbb R^{2n}}  
\left [ \partial_{x_k} u(x) \eta^2(x) -\partial_{x_k} u(x+y) \eta^2(x+y)\right ]
\left[ \partial_{x_k} u(x) -\partial_{x_k} u(x+y) \right] J(y) dx dy
\end{eqnarray*}
Substituting the above in the latter equality, we obtain  
\begin{eqnarray*}\label{}
&& \int_{\mathbb R^{n}} f'(u) |\nabla_x u(x)|^2  \eta^2(x)  dx  
=
\sum_{k=1}^n \int_{\mathbb R^{n}} |\nabla \partial_{x_k} u|^2\eta^2 dx + \frac{1}{2} \int_{\mathbb R^{n}} \nabla |\nabla u|^2 \cdot\nabla \eta^2  dx 
\\&&+   \frac{c}{2}    \iint_{\mathbb R^{2n}}   |\nabla_x u(x)|^2   \eta^2(x)  J(y) dx dy 
 +  \frac{c}{2}   \iint_{\mathbb R^{2n}}   |\nabla_x u(x+y)|^2   \eta^2(x+y) J(y) dx dy
 \\&&   -  \frac{c}{2}   \iint_{\mathbb R^{2n}}    \nabla_x u(x) \cdot \nabla_x u(x+y)   \eta^2(x)  J(y) dx dy 
    -  \frac{c}{2}   \iint_{\mathbb R^{2n}}   \nabla_x u(x) \cdot \nabla_x u(x+y)   \eta^2(x+y)  J(y) dx dy   . 
\end{eqnarray*}
From this and (\ref{Testab}), we conclude   
\begin{eqnarray*}\label{HijIden2}
 && \int_{\mathbb R^{n}}\left[ \sum_{k=1}^n  |\nabla \partial_{x_k} u|^2 -  |\nabla |\nabla  u|   |^2 \right] \eta^2 dx  +   c \iint_{\mathbb R^{2n}}     |\nabla_x u(x)| |\nabla_x u(x+y)| \eta(x) \eta(x+y) J(y) dy dx 
\\&\le&  
  \frac{c}{2}   \iint_{\mathbb R^{2n}}    \nabla_x u(x) \cdot \nabla_x u(x+y)  \left[ \eta^2(x) +\eta^2(x+y) \right]  J(y) dx dy  .
\end{eqnarray*}
We now apply recompute the left-hand side of the above inequality.    According to formula (2.1) given in  \cite{sz1,sz2}, the following geometric identity between the tangential gradients and curvatures holds. For any $w \in C^2(\Omega)$
 \begin{eqnarray*}\label{identity} \sum_{k=1}^{n} |\nabla \partial_k w|^2-|\nabla|\nabla w||^2=
\left\{
                      \begin{array}{ll}
                       |\nabla w|^2 (\sum_{l=1}^{n-1} \mathcal{\kappa}_l^2) +|\nabla_T|\nabla w||^2 & \hbox{for $x\in\{|\nabla w|>0\cap \Omega \}$,} \\
                       0 & \hbox{for $a.e.$ $x\in\{|\nabla w|=0\cap \Omega \}$,}
                                                                       \end{array}
                    \right.
                       \end{eqnarray*} 
 where $ \mathcal{\kappa}_l$ are the principal curvatures of the level set of $w$ at $x$ and $\nabla_T$ denotes the orthogonal projection of the gradient along this level set. On the other hand, it is straightforward to notice that $$ \eta(x) \eta(x+y)  =\frac{1}{2} [\eta^2(x) +\eta^2(x+y)] -  \frac{1}{2} \left[ \eta(x) -\eta(x+y) \right]^2 .$$ Applying these arguments  to the latter inequality completes the proof. 
 
\hfill $\Box$

\begin{lemma}\label{lemab}
Consider  $a,b\in\mathbb R^+$. Then, 
\begin{equation}\label{logab}
|\log b - \log a|^2 \le \frac{1}{ab} |b-a|^2. 
\end{equation}
\end{lemma}

\vspace*{.2 cm}

\noindent {\it Proof of Theorem \ref{Thmain}}. The proof is an application of Theorem \ref{ThmPoin}.  
 We  test the  Poincar\'{e} inequality (\ref{poinsysm1})  on the following standard test function 
\begin{equation}\label{testeta}
\eta (x):=\left\{
                      \begin{array}{ll}
                        \frac{1}{2}, & \hbox{if $|x|\le\sqrt{R}$,} \\
                      \frac{ \log {R}-\log {|x|}}{{\log R}}, & \hbox{if $\sqrt{R}< |x|< R$,} \\
                       0, & \hbox{if $|x|\ge R$.}
                                                                       \end{array}
                    \right.
 \end{equation} 
From the boundedness of $|\nabla_x u|$, we conclude that $ \mathcal B_y(\nabla_x u)$ is bounded and 
\begin{eqnarray}\label{esteta}
&&\int_{  \mathbb R^{n}\cap  \{|\nabla_x u|\neq 0\}}   \left(   |\nabla u|^2 \mathcal{\kappa}^2 + | \nabla_T |\nabla u| |^2  \right)\eta^2 dx +\frac{c}{2}\iint_{  \mathbb R^{2n}\cap  \{|\nabla_x u|\neq 0\}}    \mathcal A_y(\nabla_x u)  [\eta^2(x)+\eta^2(x+y)] J(y) dx dy 
\\&\le& \nonumber  \int_{\mathbb R^n} |\nabla u|^2 |\nabla \eta|^2 dx+ C \iint_{  \mathbb R^{2n}}  [ \eta(x) - \eta(y) ] ^2 J(x-y) d x dy  , 
  \end{eqnarray} 
 where $C>0$. For the rest of the proof, we provide an estimate for the right-hand side of the above inequality. First of all it is straightforward to compute  
   \begin{equation} \label{BR}
\int_{B_R\setminus B_{\sqrt {R} } }   |\nabla u|^2   |\nabla \eta|^2 \le C \left\{
                      \begin{array}{ll}
                        \frac{1}{\log R}, & \hbox{if $n=2$,} \\
                       \frac{R^{n-2}+ R^{(n-2)/2}}{|n-2||\log R|^2}, & \hbox{if $n\neq 2$.}
                                                                       \end{array}
                    \right.\end{equation} 
 Therefore, in two dimensions $\frac{C}{\log R}$ is an upper bound estimate when $C>0$ is a constant independent from $R$. Now, we provide an estimate for the second term in the right-hand side of (\ref{esteta}).  Due to the symmetry in this term, we shall   a symmetric domain decomposition 
 \begin{equation}\label{domgamm}
  \Gamma_{R,\sqrt R}:=\cup_{i=1}^6 \Gamma^i_R,
\end{equation} as in Definition \ref{dfnsym}. From the definition of test function $\eta$ we have $|\eta(x)-\eta(y)|=0$ on $\Gamma^5_R$ and $\Gamma^6_R$. In addition, if $(x,y)\in \Gamma^4_R$, then $x\in B_{\sqrt R}$ and $y\in \mathbb R^n\setminus B_{R}$. This implies that $|x-y|>R-\sqrt R$.  So, for large enough $R$, we conclude that $|x-y|>\delta_0$. Therefore, $I_4(R)=0$ for large enough $R$.  From  this, (\ref{BR}) and (\ref{esteta}), we conclude 
\begin{eqnarray}\label{intIJR1}
&&\int_{  \mathbb R^{n}\cap B_{\sqrt R}\cap  \{|\nabla_x u|\neq 0\}}     |\nabla u|^2 \mathcal{\kappa}^2 + | \nabla_T |\nabla u| |^2    dx 
+ c  \iint_{ \{\mathbb R^n \times B_{\sqrt R}\} \cap \{|\nabla_x u|\neq 0\} }    \mathcal A_y(\nabla_x u) J(y)  dx dy 
 \\&\le&  \label{intIJR2}  \frac{C}{\log R} + C \sum_{i=1}^3  \iint_{\Gamma^i_R \cap |x-y|\le \delta_0 }  \left[ \eta(x) - \eta(y) \right] ^2 |  x-y |^{-n-\alpha} d x dy  =:
\frac{C}{\log R} +  C \sum_{i=1}^3 I_i(R) . 
  \end{eqnarray} 
  We now provide an upper-bound estimate for $I_i(R)$ when $1\le i\le 3$ as follows.
\\
\\
\noindent {\bf Upper-Bound for ${\bf I_1(R)}$}. Assume that  $(x,y)\in \Gamma^1_R \cap \{|x-y|\le \delta_0\}$ which implies,  without loss of generality, 
$$x\in B_{\sqrt{R}}\setminus  B_{\sqrt{R}-\delta_0}\ \ \text{ and } \ y\in B_ {\sqrt{R}+\delta_0} \setminus B_{\sqrt R}.$$ 
Therefore, $\eta(x)=\frac{1}{2}$ and $\eta(y)=1-\frac{\log |y|}{\log R}$. From (\ref{logab}) and the fact that $|x|< \sqrt R \le |y|$,  we get 
\begin{eqnarray*}
|\eta(x) -\eta(y)|^2 &=& \frac{1}{\log^2 R} |\log |y| - \log \sqrt R|^2
\le \frac{1}{\log^2 R} \frac{1}{|y| \sqrt R}  | |y| -  \sqrt R|^2 
 \le \frac{1}{R\log^2 R}   | |y| -  |x||^2 
 \\&\le& \frac{1}{R\log^2 R}   | y -  x|^2   .
\end{eqnarray*}
Therefore, 
\begin{eqnarray*}
I_1(R) 
\le  \frac{C }{R\log^2 R} \left[ \int_{B_{\sqrt R}\setminus B_{\sqrt R-\delta_0}} dx \right]\left[\int_{B_{\delta_0}}  |z|^{2-n-\alpha} dz \right] \le   \frac{C}{\sqrt R \log^2 R} , 
\end{eqnarray*} 
where $C=C( \frac{\delta_0^{2-\alpha}}{2-\alpha})$ is positive constant independent from $R$. Here, we have used the assumptions $\alpha<2$ and $n=2$. 
\\
\\
\noindent {\bf Upper-Bound for ${\bf I_2(R)}$}.  Suppose that  $(x,y)\in \Gamma^2_R \cap \{|x-y|\le \delta_0\}$ which implies,  without loss of generality,  $|x|\le |y|$. From (\ref{logab})  and since $x,y\in B_ R\setminus B_{\sqrt R}$,  we have 
\begin{equation*}
|\eta(x) -\eta(y)|^2 = \frac{1}{\log^2 R} |\log |y| - \log |x||^2 \le \frac{1}{\log^2 R} \frac{1}{|x| |y|}  | |y| -  |x||^2 \le \frac{1}{|x|^2\log^2 R}   | y -  x|^2   .
\end{equation*}
Therefore, 
\begin{eqnarray*}
I_2(R) &\le& 
\frac{C}{\log^2 R} \left[ \int_{B_{ R}\setminus B_{\sqrt R}} \frac{1}{|x|^2} dx \right]\left[  \int_{B_{\delta_0}}  |z|^{2-n-\alpha} dz  \right]
\le \frac{C }{\log^2 R} \left[ \int_{\sqrt R}^R r^{n-3}dr \right] \left[\int_{0}^{\delta_0}  r^{1-\alpha} dr\right] 
\\&\le &  \frac{C}{\log R} , 
\end{eqnarray*} 
where again $C=C( \frac{\delta_0^{2-\alpha}}{2-\alpha})$ is positive constant independent from $R$, and $n=2$. 
\\
\\
\noindent {\bf Upper-Bound for ${\bf I_3(R)}$}.  Suppose that $(x,y)\in \Gamma^3_R \cap \{|x-y|\le \delta_0\}$, which implies, without loss of generality, $x\in B_{R}\setminus  B_{R-\delta_0}$ and $y\in B_ {{R}+\delta_0} \setminus B_{R}$ for large enough $R$.  Therefore,  $\eta(x)=1-\frac{\log |x|}{\log R}$ and $\eta(y)=0$. Applying (\ref{logab}) and the fact that $|x|<  R \le |y|$,  we have  
\begin{eqnarray*}
|\eta(x) -\eta(y)|^2 &=& \frac{1}{\log^2 R} |\log |x| - \log  R|^2 \le \frac{1}{\log^2 R} \frac{1}{|x| R}  | |x| -   R|^2  \le \frac{1}{|x|^2 \log^2 R}   | |y| -  |x||^2 
\\&\le& \frac{1}{|x|^2\log^2 R}   | y -  x|^2  . 
\end{eqnarray*}
Therefore, 
\begin{equation*}
I_3(R) \le 
\frac{C}{\log^2 R}  \left[\int_{B_{ R}\setminus B_{R-\delta_0}} \frac{1}{|x|^2} dx \right]\left[  \int_{B_{\delta_0}}|z|^{2-n-\alpha} dz \right]
\le   \frac{C}{\log^2 R} . 
\end{equation*} 
where again $C=C( \frac{\delta_0^{2-\alpha}}{2-\alpha})$ is positive constant independent from $R$, and $n=2$.  

\noindent Combining the above upper-bounds and (\ref{intIJR1})-(\ref{intIJR2}), we conclude that  for large $R$
\begin{equation}\label{intIJR3}
\int_{  \mathbb R^{2}\cap B_{\sqrt R}\cap  \{|\nabla_x u|\neq 0\}}     |\nabla u|^2 \mathcal{\kappa}^2 + | \nabla_T |\nabla u| |^2    dx 
+ c  \iint_{ \{\mathbb R^2 \times B_{\sqrt R}\} \cap \{|\nabla_x u|\neq 0\} }    \mathcal A_y(\nabla_x u) J(y)  dx dy  \le  
\frac{C}{\log R} .
  \end{equation} 
  Note that for all $x,y\in\mathbb R^2$  $$\mathcal A_y(\nabla_x u)=  |\nabla_x u(x)|  |\nabla_x u(x+y)| -\nabla_x u(x)  \cdot \nabla_x u(x+y) \ge 0.$$ Sending $R\to\infty$, and assuming $|\nabla_x u|\neq 0$, we get
 \begin{equation*}
  |\nabla u|^2 \mathcal{\kappa}^2=0, \ \  | \nabla_T |\nabla u| |^2  =0 \ \ \text{and} \ \  \mathcal A_y(\nabla_x u) J(y) = 0 \ \ \text{a.e. for all} \ \  x,y\in\mathbb R^2. 
  \end{equation*}
Therefore,  $\mathcal A_y(\nabla_x u) = 0$ for all $x\in\mathbb R^2$ and $y\in B_{\delta_1}\subset \mathbb R^2$. This implies that 
\begin{equation*}
|\nabla_x u(x)|  |\nabla_x u(x+y) |=\nabla_x u(x)  \cdot \nabla_x u(x+y) , 
\end{equation*}
when $|\nabla_x u|\neq 0$ that is equivalent to 
\begin{equation*}
u_{x_1}(x) u_{x_2}(x+y)= u_{x_1}(x+y) u_{x_2}(x), 
\end{equation*}
and therefore, 
  \begin{equation*}\label{unabla}
\nabla_x u(x) \cdot \nabla_x^\perp u(x+y)=0. 
\end{equation*}
This completes the proof.    

\hfill $\Box$

\vspace*{.2 cm}

\noindent {\it Proof of Theorem \ref{Thmain2}}.  The proof is similar to the proof of Theorem \ref{ThmPoin}.  We apply the test function (\ref{testeta}) in (\ref{esteta}). The jump kernel $J$ satisfies (\ref{JDecay}) with a decay rate (\ref{JDecayr}) when $D(r)<C r^{-\theta}$ for $\theta>3$. From the definition of test function $\eta$ we have $|\eta(x)-\eta(y)|=0$ on $\Gamma^5_R$ and $\Gamma^6_R$. From the estimate (\ref{BR}) and the domain decomposition (\ref{domgamm}), 
\begin{eqnarray}\label{intIJR2}
&&\int_{  \mathbb R^{n}\cap B_{\sqrt R}\cap  \{|\nabla_x u|\neq 0\}}     |\nabla u|^2 \mathcal{\kappa}^2 + | \nabla_T |\nabla u| |^2    dx 
+ c  \iint_{ \{\mathbb R^n \times B_{\sqrt R}\} \cap \{|\nabla_x u|\neq 0\} }    \mathcal A_y(\nabla_x u) J(y)  dx dy 
 \\&\le&  \label{intIJR2}  \frac{C}{\log R} + 
 C \sum_{i=1}^4  \iint_{\Gamma^i_R \cap |x-y|\le \delta_0 }  \left[ \eta(x) - \eta(y) \right] ^2 |  x-y |^{-n-\alpha} d x dy \\
&&
 +C \sum_{i=1}^4  \iint_{\Gamma^i_R \cap |x-y| > \delta_0 }  \left[ \eta(x) - \eta(y) \right] ^2  J(x-y) d x dy  =:
\frac{C}{\log R} +  C \sum_{i=1}^4 I_i(R) + C\sum_{i=1}^4 \bar I_i(R) . 
  \end{eqnarray} 
Note that upper-bound estimates for all $I_i(R)$ are given in the proof of Theorem \ref{Thmain}. We now establish upper-bounds for  $\bar I_i(R)$ for $1\le i\le 4$ in various cases. 
\\
\\
\noindent {\bf Upper-Bound for ${\bf \bar I_1(R)}$}. Assume that  $(x,y)\in \Gamma^1_R \cap \{|x-y| >\delta_0\}$. In this case, we have 
\begin{equation*}
|\eta(x) -\eta(y)|^2 \le \frac{1}{R\log^2 R}   | y -  x|^2   .
\end{equation*}
For the algebraic decay $D(r)< C r^{-\theta}$ when $\theta>3$, and for two dimensions, we have   
 \begin{eqnarray}\label{J1R}
\bar I_1(R) & \le & \frac{C}{R\log^2 R} \left[ \int_{ B_{\sqrt{R}} } dx\right] \left[ \sum_{k=1}^\infty \int_{k \delta_0 <|z|<2k \delta_0} |z|^2 J(z) dz \right]
\\&\le&\nonumber \frac{C \delta_0^{2-\theta}}{R\log^2 R}   \left[ \int_{ B_{\sqrt{R}} }  dx \right] \left[ \sum_{k=1}^\infty k^{2-\theta}\right] \le  \frac{C }{ \log^2 R} , 
\end{eqnarray} 
when $C=C(\delta_0^{2-\theta } )$ is positive constant independent from $R$.  
\\
\\
\noindent {\bf Upper-Bound for ${\bf \bar I_2(R)}$}. Suppose that  $(x,y)\in \Gamma^2_R \cap \{|x-y|> \delta_0\}$. In this case, assuming that $|x|\le |y|$, we have 
\begin{equation*}
|\eta(x) -\eta(y)|^2  \le \frac{1}{|x|^2\log^2 R}   | y -  x|^2   .
\end{equation*}
For the jump kernel satisfying (\ref{JDecay}) and (\ref{JDecayr}) with the decay rate $D(r)< C r^{-\theta}$ when $\theta>3$,  we have 
\begin{eqnarray}\label{J2R}
\bar I_2(R) &\le&  \frac{C}{\log^2 R} \left[ \int_{B_{ R}\setminus B_{\sqrt R}} \frac{dx}{|x|^2} \right] \left[ \sum_{k=1}^\infty \int_{k \delta_0<|z|<2k \delta_0} |z|^2 J(z) dz\right]
\\&\le& \nonumber \frac{C \delta_0^{2-\theta}}{\log^2 R} \left[  \int_{\sqrt R}^R r^{n-3}dr  \right]  \left[\sum_{k=1}^\infty k^{2-\theta}\right] \le  \frac{C }{\log R}. 
\end{eqnarray} 
 when $C=C(\delta_0^{2-\theta } )$ is positive constant independent from $R$.  
\\
\\
\noindent {\bf Upper-Bound for ${\bf \bar I_3(R)}$}. Suppose that $(x,y)\in \Gamma^3_R \cap \{|x-y|> \delta_0\}$.  Since $\eta(x)=1-\frac{\log |x|}{\log R}$ and $\eta(y)=0$, for $|x|<  R \le |y|$, we have 
\begin{equation*}
|\eta(x) -\eta(y)|^2  \le \frac{1}{|x|^2\log^2 R}   | y -  x|^2  . 
\end{equation*}
Since this is the same as the previous case, we get the same upper-bound that is 
\begin{equation}\label{J3R}
\bar I_3(R) \le   \frac{C}{\log R}.
\end{equation} 
\\
\noindent {\bf Upper-Bound for ${\bf \bar I_4(R)}$}. Suppose that  $(x,y)\in \Gamma^4_R \cap \{|x-y|> \delta_0\}$. In this case, we have $\eta(x)=\frac{1}{2}$ and $\eta(y)=0$ and $|x-y|>R-\sqrt R>\delta_0$ for  large enough $R$.   For the jump kernel satisfying (\ref{JDecay}) and (\ref{JDecayr}) with the decay rate $D(r)< C r^{-\theta}$ when $\theta>3$,  we have 
\begin{equation*}\label{J4R}
\bar I_4(R) = \frac{1}{4} \int_{ B_{\sqrt R}}  dx   \sum_{k=1}^\infty \int_{k(R-\sqrt R)<|z|<2k (R-\sqrt R)} J(z) dz \le \frac{C R}{(R-\sqrt R)^{\theta }}  \sum_{k=1}^\infty k^{-\theta } 
 \le  \frac{C}{ R^{\theta -1} }. 
\end{equation*} 
 Combining the above cases, we conclude that  for large $R$
\begin{equation}\label{intIJR4}
\int_{  \mathbb R^{2}\cap B_{\sqrt R}\cap  \{|\nabla_x u|\neq 0\}}     |\nabla u|^2 \mathcal{\kappa}^2 + | \nabla_T |\nabla u| |^2    dx 
+ c  \iint_{ \{\mathbb R^2 \times B_{\sqrt R}\} \cap \{|\nabla_x u|\neq 0\} }    \mathcal A_y(\nabla_x u) J(y)  dx dy  \le  
\frac{C}{\log R} .
  \end{equation} 
This completes the proof. 

\hfill $\Box$

\vspace*{.2 cm}

\noindent {\it Proof of Theorem \ref{thmlione}}. Let  $\eta\in C_c^1(\mathbb R^n)$ be a  test function.  Multiply both sides of (\ref{linPhi}) with $\eta^2(x) \sigma(x) $ and integrate to get  
 \begin{equation*}\label{}
-\int_{\RR^{n}} \div(\phi^2(x)\nabla \sigma(x)) \eta^2(x) \sigma(x)  dx+ \iint_{\RR^{2n}} \left( \sigma(x)- \sigma(y) \right)  \sigma(x) \phi(x) \phi(y)  J(x-y) \eta^2 (x) dx dy  = 0 . 
\end{equation*}
This yields 
    \begin{equation*}\label{}
\int_{\RR^{n}} \phi^2(x)\nabla \sigma(x) \cdot\nabla ( \eta^2(x) \sigma(x) ) dx  + \iint_{\RR^{2n}}  [\eta^2(x) \sigma(x) -  \eta^2(y) \sigma(y)]    [\sigma(x)- \sigma(y)] \phi(x) \phi(y)  J(x-y)   dy dx = 0   . 
  \end{equation*}
 Simplifying the above and using the following formula 
\begin{equation}\label{etaiden}
 [\eta^2(x) \sigma(x) -  \eta^2(y) \sigma(y)] = \frac{1}{2}  [\sigma(x) -  \sigma(y)][\eta^2(x) + \eta^2(y)] 
 + \frac{1}{2}  [\sigma(x) +  \sigma(y)][\eta^2(x) - \eta^2(y)], 
\end{equation}
we conclude 
\begin{eqnarray*}
0&=&I^2(\mathbb R^n)+K^2 ({\RR^{2n}} )\\&&+2  \int_{\RR^{n}} \phi^2  \sigma \eta \nabla \sigma \cdot\nabla  \eta  dx + \frac{1}{2}\iint_{\RR^{2n}}   [\sigma^2(x)- \sigma^2(y)] \phi(x) \phi(y)   [\eta^2(x) - \eta^2(y)] J(x-y)   dy dx , 
\end{eqnarray*}
for $\Omega_1\subseteq \mathbb R^n$ and $\Omega_2\subseteq \mathbb R^{2n}$, and 
\begin{eqnarray*}
I^2(\Omega_1)&:=&\int_{\Omega_1} \phi^2  |\nabla \sigma|^2 \eta^2   dx , 
 \\
 K^2(\Omega_2)&:=&\frac{1}{2}\iint_{\Omega_2}   [\sigma(x)- \sigma(y)]^2 \phi(x) \phi(y)   [\eta^2(x)+ \eta^2(y)] J(x-y)   dy dx  .
 \end{eqnarray*}
  Applying the Cauchy-Schwarz inequality, one can see that 
\begin{equation}
I^2(\mathbb R^n)+K^2(\mathbb R^{2n}) \le C [I(\mathbb R^n) M (\mathbb R^n) + K(\mathbb R^{2n})N(\mathbb R^{2n})] ,
\end{equation}
when $C$ is a nonnegative constant and 
\begin{eqnarray*}
M(\Omega_1) &:=& \int_{\Omega_1} \phi^2 \sigma^2 |\nabla \eta |^2    dx  ,
\\
 N(\Omega_2)&:=&\frac{1}{2}\iint_{\Omega_2}   [\sigma(x) + \sigma(y)]^2 \phi(x) \phi(y)   [\eta(x)- \eta(y)]^2 J(x-y)   dy dx . 
  \end{eqnarray*}
Let $\eta=1$ in $\overline {B_R}$ and $\eta=0$ in $\overline{\RR^n\setminus B_{2R}}$ with $||\nabla \eta||_{L^{\infty}(B_{2R}\setminus B_R)}\le C R^{-1}$. Then, for $R>1$
\begin{equation}
I^2(\mathbb R^n)+K^2(\mathbb R^{2n}) \le C [I(\mathbb R^n) M (B_{2R}\setminus B_R) + K(\mathbb R^{2n})N(\cup_{k=1}^4 \Gamma^k_R)]. 
\end{equation}
From the fact that $||\nabla \eta||_{L^{\infty}(B_{2R}\setminus B_R)}\le C R^{-1}$ and the assumption (\ref{Phiuxuy}) we get $M (B_{2R}\setminus B_R)\le C R^2$. On ther other hand, for $(x,y)$ in $\{\cup_{k=1}^4\Gamma^k_R \} $  we have 
  \begin{equation*}
 (\eta(x)-\eta(y))^2 \le C R^{-2} |x-y|^2  .    
   \end{equation*}
From this and the assumption (\ref{Phiuxuy}), we conclude  
 \begin{equation*}
N(\cup_{k=1}^4 \Gamma^k_R)  \le C R^{-2}  \iint_{\cup_{k=1}^4 \Gamma^k_R}   [\sigma(x) +  \sigma(y)]^2  \phi(x) \phi(y) |x-y|^2 J(x-y)   dy dx \le C. 
   \end{equation*}
Therefore, 
\begin{equation}
I^2(\mathbb R^n)- C I(\mathbb R^n) + K^2(\mathbb R^{2n}) -C K(\mathbb R^{2n}) \le 0. 
\end{equation}
Since $I,K\ge 0$, $I,K$ are bounded. Therefore, $0\le I(B_R)\le C$ and $0\le K(B_R\times B_R)\le C$ when $C$ is independent from $R$. This implies that  $I,K\equiv 0$. This implies $\sigma$ must be a constant.

 \hfill $\Box$

 \begin{lemma} Let $g,h\in C^1(\mathbb R^n)$, then
\begin{eqnarray}\label{idenLL}
L [g(x)h(x)] &=& g(x) L [h(x)] + h(x) L[g(x)]
\\&& \nonumber -   \int_{\mathbb R^n}    \left  [g(x) - g(y)  \right]  \left[h(x)-h(y)  \right] J(x-y) dy. 
\end{eqnarray}
\end{lemma}

\begin{prop}\label{proplin}
Let $\phi$ and $\psi$ be classical solutions for the linearized equation (\ref{main}) that is 
\begin{eqnarray}\label{linphi1}
\Delta \phi + c L[\phi] = f'(u) \phi \  \ \text{and} \ \  
\label{linphi2} \Delta \psi + c L[\psi] = f'(u) \psi \ \ \text{in} \ \ \mathbb R^n. 
 \end{eqnarray}
Let $\phi>0$ and define the quotient $\sigma:=\frac{\psi}{\phi}$. Then, 
\begin{equation}\label{linPhip}
\div(\phi^2(x)\nabla \sigma(x)) +    \lim_{\epsilon\to 0} \int_{\{y\in \mathbb R^n, |x-y|>\epsilon\} }  \left( \sigma(y)- \sigma(x) \right)\phi(x) \phi(y)  J(x-y) dy  = 0 \ \ \text{in} \ \ \mathbb R^n . 
\end{equation}
\end{prop}
\begin{proof}
Since $\psi =\sigma \phi$,  we have 
\begin{equation}\label{Ls}
\Delta (\sigma (x) \phi (x))+ c L[ \sigma (x) \phi (x)] = f'(u) \sigma(x) \phi(x). 
\end{equation}
Multiplying (\ref{linphi1}) with $\sigma $ and combining with (\ref{Ls}) we get 
\begin{equation}\label{LLL}
\Delta (\sigma (x) \phi (x))  - \sigma (x) \Delta (\phi (x)) +c\left( L[  \sigma(x) \phi(x) ] -  \sigma(x) L[\phi(x)]\right)= 0 . 
\end{equation}
Applying formula (\ref{idenLL}) and 
$$ 
\phi (x) \Delta (\sigma (x))+2 \nabla \sigma (x)\cdot\nabla\phi (x)  +c \lim_{\epsilon\to 0} \int_{\{y\in \mathbb R^n, |x-y|>\epsilon\} }   [\sigma(y)-\sigma(x)] \phi(y) J(x-y) dy =0. 
$$
Multiplying $\phi(x)$ completes the proof. 

\end{proof}

Let $u$ be a monotone solution of (\ref{main}). Set $\phi:=\frac{\partial u}{\partial x_n} $ and $\psi:=\nabla u\cdot \nu$ for $\nu(x)=\nu(x',0):\RR^{n-1}\to \RR $. Therefore, $\phi$ and $\psi$ satisfy the linearized equation that is (\ref{linphi1})-(\ref{linphi2}). Now, define the quotient $\sigma:=\frac{\psi}{\phi}$. From Proposition \ref{proplin}, we have 
\begin{equation}\label{linPhip2}
\div(\phi^2(x)\nabla \sigma(x)) +    \lim_{\epsilon\to 0} \int_{\{y\in \mathbb R^n, |x-y|>\epsilon\} }  \left( \sigma(y)- \sigma(x) \right)\phi(x) \phi(y)  J(x-y) dy  = 0 \ \ \text{in} \ \ \mathbb R^n . 
\end{equation}
 Since $|\nabla u|$ is globally bounded,  we conclude that $|\sigma|\le \frac{C}{\phi}$. This implies that 
\begin{equation*}
 [\sigma(x) +  \sigma(y)]^2 \le C \left(\frac{1}{\phi^2(x)} + \frac{1}{\phi^2(y)} \right). 
\end{equation*}
Therefore, 
\begin{equation*}
 [\sigma(x) +  \sigma(y)]^2  \phi(x) \phi(y)  \le C  \left( \frac{\phi(x)}{\phi(y)} + \frac{\phi(y)}{\phi(x)}  \right). 
\end{equation*}
Suppose now that the operator $T$ satisfies the following  Harnack inequality. More precisely, let $\chi$ is continuous and positive in $\mathbb R^n$ and is a weak solution to $T\chi + a(x) \chi =0$ in $B_R$, when $a\in L^\infty(B_1)$ and $||a||_{L^\infty(B_R)}<K$, then 
 \begin{equation}\label{chiharn}
 \sup_{B_{R/2}} \chi  \le C \inf_{B_{R/2}} \chi, 
 \end{equation}
when $C$ is a positive constant depending on operator $T$ and $K$ and independent from $\chi$. Applying the above, for $\phi$ we have 
 \begin{equation}\label{harn}
 \sup_{B_1(x_0)} \phi  \le C \inf_{B_1(x_0)} \phi, \ \ \text{for all} \ \ x_0\in\mathbb R^n. 
 \end{equation}
  This implies that 
\begin{equation*}
 [\sigma(x) +  \sigma(y)]^2  \phi(x) \phi(y)  \le C. 
\end{equation*}
From this, the assumption (\ref{Phiuxuy}) in Theorem \ref{thmlione} is bounded by
\begin{equation}\label{Phiuxuysim}
C \int_{B_{2R}\setminus B_R}  dx +C   \iint_{ \{\cup_{k=1}^4 \Gamma^k_R \}}  |x-y|^2 J(x-y)   dy dx.  
 \end{equation}
 Applying similar arguments as in the proof of Theorem \ref{Thmain}, one can conclude that the above term is bounded by $CR^2$ in two dimensions.  So, Theorem \ref{thmlione} implies that $\sigma$ must be a constant. This provides a second proof for Theorem \ref{Thmain}.

 \begin{dfn}\label{dfnps} A solution $u$ of (\ref{main}) is called pointwise-stable if there exists $\chi>0$ such that solves the  linearized equation that is
 \begin{equation}\label{lin}
\Delta \chi + c L[\chi] + f'(u)\chi=0 \ \ \ \text{in } \ \ \mathbb R^n. 
 \end{equation} 
 \end{dfn}
 
 We now show that both notations of stability, the variational stability and the non-variational pointwise-stability are equivalent for solutions of (\ref{main}).  We shall follow methods and ideas  provided by  Ghoussoub and Gui in \cite{gg1},  by Berestycki, Caffarelli and Nirenberg in \cite{bcn}) and by Hamel et al. in \cite{hrsv}. 
 
 \begin{thm}\label{thmst}
 A solution of (\ref{main})  is pointwise-stable if and only if 
 it is a stable solution.
\end{thm}
\begin{proof}   If $u$ is a  pointwise-stable solution  of (\ref{main}),  from Proposition \ref{Propstab},  $u$ is stable. We now assume that  the stability inequality (\ref{stability}) holds. Let the space $H_J(\mathbb R^n)$ be defined as the closure of $C_0^\infty(\mathbb R^n)$ with the norm $||\cdot||^2_{H_J(\mathbb R^n)} := \mathcal I(\cdot,\cdot)$ for $\mathcal I$ in (\ref{mathcalI}). For $R>1$ and for $\eta\in C_0^\infty(\mathbb R^n)$,  define 
\begin{equation}
\mathcal P_R(\eta)= \frac{1}{2} \int_{B_R} |\nabla \eta(x)|^2 dx + \frac{c}{2} \iint_{\mathbb R^{n}\times \mathbb R^{n}}  | \eta(x) -\eta(y) |^2 J (x-y) dy dx - \int_{B_R} f'(u)\eta^2 dx . 
\end{equation}
Assume that $\lambda_1(R)$ is the infimum of $\mathcal P_R$ on the class of $ \Xi_R$ that is 
\begin{equation}
 \Xi_R:=\left\{ \eta\in H_J(\mathbb R^n) \ \ \text{such that} \ \ \eta=0\ \ \text{in}\ \  \mathcal C B_R \ \ \text{and}\ \  \int_{B_R} \eta^2=1\right\}. 
\end{equation}
Since $u$ is a stable solution, we have that $\lambda_1(R)\ge 0$ and there exists eigenfunction $\zeta_R$ such that the infimum is attained for a function $\zeta_R\in  \Xi_R$. Note that if $\zeta_R$ is minimizer then $|\zeta_R|$ is also a minimizer. Therefore, $\zeta_R\ge 0$. The function $\zeta_R$ is nonzero and it satisfies 
 \begin{equation}\label{zetar}
  \left\{ \begin{array}{ll}
                                            T[ \zeta_R] = f'(u) \zeta_R+ \lambda_1(R) \zeta_R, & \hbox{if $|x|< R$,} \\
                       \zeta_R=0, & \hbox{if $|x| \ge R$,}
                                                                       \end{array}
                    \right.\end{equation}
where $T= -\Delta -cL$.  From the strong maximum principle for jump-diffusion processes, we conclude that $\zeta_R>0$ in $B_R$. In addition, for $R_2>R_1$, from Lemma \ref{lemtech} and the fact that $\zeta_{R_1}=0$ in $B_{R_2}\setminus B_{R_1}$ we conclude that 
\begin{equation}
\int_{B_{R_2}} \zeta_{R_1} T [\zeta_{R_2}] = \int_{B_{R_2}} \zeta_{R_2} T[\zeta_{R_1} ]<  \int_{B_{R_1}} \zeta_{R_2} T[\zeta_{R_1}]. 
\end{equation}
Applying this argument to solutions $\zeta_{R_1}$ and $\zeta_{R_2}$, we conclude that 
\begin{equation} 
\lambda_1(R_2) \int_{B_{R_1}} \zeta_{R_1} \zeta_{R_2} < \lambda_1(R_1) \int_{B_{R_1}} \zeta_{R_1} \zeta_{R_2}. 
\end{equation}
This implies that $\lambda_1(R)$ is decreasing in $R$. Therefore, $\lambda_1(R)>0$ for any $R>1$.  
We now consider the elliptic problem 
\begin{equation}\label{chiR}
  \left\{ \begin{array}{ll}
                                            T[ \chi_R] = f'(u) \chi_R, & \hbox{if $|x|< R$,} \\
                       \chi_R=m_R, & \hbox{if $|x| \ge R$,}
                                                                       \end{array}
                    \right.\end{equation}
where $m_R$ is a fixed positive constant. Considering $\phi_R=\chi_R-m_R$,  the above problem is connected with 
\begin{equation}\label{phiR}
  \left\{ \begin{array}{ll}
                                            T[ \phi_R] = f'(u) \phi_R +  f'(u) m_R , & \hbox{if $|x|< R$,} \\
                       \phi_R=0, & \hbox{if $|x| \ge R$.}
                                                                       \end{array}
                    \right.\end{equation}
This implies that $\chi_R$ and $\phi_R$ exist. Now, multiply (\ref{chiR}) with $\chi^-_R$ and integrate to conclude 
\begin{eqnarray}\label{}
&& \int_{B_R} \nabla \chi_R(x) \cdot \nabla \chi^-_R(x) dx + 
 \frac{c}{2} \iint_{\mathbb R^{n}\times \mathbb R^{n}}  ( \chi_R(x) -\chi_R(y) )( \chi_R^-(x) -\chi_R^-(y) ) J (x-y) dy dx 
\\ &&= \int_{B_R} f'(u) \chi_R(x)  \chi_R^-(x) dx  = - \int_{B_R} f'(u)   |\chi_R^-(x)|^2 dx .
\end{eqnarray}
Note that 
\begin{equation}
( \chi_R(x) -\chi_R(y) )( \chi_R^-(x) -\chi_R^-(y) ) \le - ( \chi_R(x) -\chi_R(y) )^2 \ \ \text{and} \ \ 
\nabla \chi_R(x) \cdot \nabla \chi^-_R(x) \le -|\nabla \chi_R(x)|^2. 
\end{equation}
This implies that 
\begin{equation}
\mathcal P_R(\chi^-_R) = \frac{1}{2} \int_{B_R} |\nabla \chi^-_R(x)|^2 dx + \frac{c}{2} \iint_{\mathbb R^{n}\times \mathbb R^{n}}  | \chi^-_R(x) -\chi^-_R(y) |^2 J (x-y) dy dx - \int_{B_R} f'(u)|\chi^-_R|^2 dx
\le 0. 
\end{equation}
From this we conclude that $\chi^-_R\equiv 0$ that is $\chi_R\ge 0$.   From some standard elliptic estimates, there is a subsequence $\{R_k\}_k$ going to infinity that $\chi_{R_k}$  converges $\chi>0$ that satisfies the linearized equation (\ref{lin}). This completes the proof.

\end{proof}

\section{Energy Estimates; Proofs of Theorem \ref{thmene}-\ref{thmeneal}}\label{secen}

In this section, we provide proofs for the energy estimates provided as main results. 
\\
\\
\noindent {\it Proof of Theorem \ref{thmene}}.   Set $c= 1$. Define the shift function $u^t(x):=u( x',x_n+t)$ for $( x',x_n)\in\mathbb R^{n-1}\times \mathbb R$ and $t\in\mathbb R$.   The energy functional for the shift function $u^t$ is 
\begin{equation}\label{energyt}
\mathcal E(u^t,B_R) =\mathcal E^{\text{Sob}}{(u^t,B_R)} + \mathcal E^{\text{Pot}} {(u^t,B_R)} , 
\end{equation}
  for  $R>\delta_0$ and 
\begin{eqnarray}\label{Esob}
\mathcal E^{\text{Sob}}(u^t,B_R) &=&  \frac{1}{2} \int_{B_R} |\nabla u^t(x)|^2 dx + 
\frac{1}{4} \int_{B_R} \int_{B_R}   [  u^t(x) -u^t(y) ]^2 J(x-y) dy dx \\&&+  \frac{1}{2} \int_{B_R} \int_{\mathbb{R}^n\setminus B_R}   [  u^t(x) -u^t(y) ]^2  J(x-y) dy dx, 
\end{eqnarray}
 and 
\begin{equation}\label{Epot}
\mathcal E^{\text{Pot}} {(u^t,B_R)}  =\int_{\Omega}  F(u^t(x)) dx. 
\end{equation}
 We now differentiate the energy functional  in terms of parameter $t$ to get
\begin{eqnarray*}
\partial_t\mathcal E(u^t,B_R) &=&   \int_{B_R} \nabla u^t(x)\cdot \nabla \partial_t u^t(x)  dx + \frac{1}{2} \int_{B_R} \int_{B_R}  [  u^t(x) -u^t(y) ]  [  \partial_t u^t(x) -\partial_t u^t(y) ]  J(x-y) dy dx \\&&+   \int_{B_R} \int_{\mathbb{R}^n\setminus B_R} [  u^t(x) -u^t(y) ]  [ \partial_t u^t(x) 
- \partial_t u^t(y) ] J(x-y) dy dx  - \int_{B_R}  f (u^t) \partial_t u^t dx   . 
\end{eqnarray*}
From Lemma \ref{lemtech} and performing integrating by parts, we conclude 
\begin{eqnarray*} 
\partial_t\mathcal E(u^t,B_R) &=&  
  \int_{\partial B_{R}} \partial_\nu u^t \partial_t u^t  dx +   \int_{\mathbb{R}^n\setminus B_R} \int_{B_R}  [  u^t(x) -u^t(y) ]   \partial_t u^t(x)  J(x-y) dy dx  \\&& \label{tutT}+ \int_{B_R} \partial_tu^t(x) \left( -\Delta u^t(x) - L[u^t(x)] \right) dx  - \int_{B_R}  f (u^t) \partial_t u^t dx   .
\end{eqnarray*}
Since $u^t$ is a solution of (\ref{main}), we can simplify the above as 
\begin{equation}
\partial_t \mathcal E(u^t,B_R) =   \int_{\partial B_{R}} \partial_\nu u^t \partial_t u^t  dx +  \int_{\mathbb{R}^n\setminus B_R} \int_{B_R} [  u^t(x) -u^t(y) ]   \partial_t u^t(x)  J(x-y) dy dx    .
\end{equation}
Since $|\partial_\nu u^t|\le M$ and $\partial_t u^t >0$, we get 
\begin{equation}
\partial_t \mathcal E(u^t,B_R) \ge  -M \int_{\partial B_{R}}  \partial_t u^t  dx +  \int_{\mathbb{R}^n\setminus B_R} \int_{B_R} [  u^t(x) -u^t(y) ]   \partial_t u^t(x)  J(x-y) dy dx    .
\end{equation}
Note that  $ \mathcal E(u,B_R)= \mathcal E(1,B_R)- \int_0^\infty \partial_t \mathcal E(u^t,B_R) dt$.  From the fact that $ \mathcal E(1,B_R)=0$, we obtain   
\begin{equation*}\label{EKT}
\mathcal E(u, B_R) \le  CR^{n-1} + \int_{\mathbb{R}^n\setminus B_R} \int_{B_R}  \int_0^\infty |   u^t(x) -u^t(y)  |  \partial_t u^t(x)  J(x-y) dt dy dx .  
\end{equation*}
  Note that $|u^t(x)- u^t(y)| \le C |x-y|$. From the boundedness of $u$ and $|\nabla u|$,   we have  
\begin{equation}\label{EUBR}
\mathcal E(u, B_R) \le  CR^{n-1} + C   \iint_{[(\mathbb{R}^n\setminus B_R )\times B_R ] }  |x-y|  J(x-y)  dy dx   . 
\end{equation}
We now apply a domain decomposition for $\Omega_R:=(\mathbb{R}^n\setminus B_R )\times B_R$ that is $\Omega_R=\cup_{i=1}^3 \Omega^i_R$ and  
\begin{equation}\label{Pi123}
\Omega^1_R:= (\mathbb{R}^n\setminus B_{R+\delta_0}) \times B_{R}, \ \ \Omega^2_R:= (B_{R+\delta_0}\setminus B_R) \times B_{R-\delta_0}, \ \ \Omega^3_R:= (B_{R+\delta_0}\setminus B_{R}) \times (B_{R}\setminus B_{R-\delta_0}). 
\end{equation}
Since the jumping kernel $J$ is truncated, $J$ is identically vanishes on $\Omega^1_R$ and $\Omega^2_R$. Therefore, the above estimate can be reformulated as 
 \begin{equation}\label{Bdelta}
\mathcal E(u, B_R) \le CR^{n-1} +  C   \iint_{ (B_{R+\delta_0}\setminus B_{R}) \times (B_{R}\setminus B_{R-\delta_0})  }  |x-y|  J(x-y)  dy dx   . 
\end{equation}
Hence, 
\begin{equation}\label{EphiBR}
\mathcal E(u, B_R) \le   CR^{n-1} +  C  \int_{ B_{R}\setminus B_{R-\delta_0}   } \int_{  B_{R+\delta_0}\setminus B_{R}  }   |x-y|^{1-n-\alpha}   dy dx   . 
\end{equation}
It is straightforward computations to show that 
\begin{equation}\label{IntIER}
  \int_{ B_{R}\setminus B_{R-\delta_0}   } \int_{  B_{R+\delta_0}\setminus B_{R}  }   |x-y|^{1-n-\alpha}   dy dx 
\le  C \left\{ \begin{array}{lcl}
\hfill \delta_0 R^{n-1}    \ \ \text{for}\ \ \  \alpha=1,\\   
\hfill \frac{(2\delta_0)^{2-\alpha}}{(1-\alpha)(2-\alpha)} R^{n-1}    \ \ \text{for}\ \ \alpha \neq 1. 
\end{array}\right.
\end{equation}
Here,  $C$ is a positive constant it does not depend on $R,\alpha, \delta_0$.  Combining (\ref{IntIER}) and (\ref{EphiBR}) finishes the proof of (\ref{EKRnminus1}) for the truncated kernels satisfying (\ref{Jumpki}). 

Now, assume that the kernel $J$ has decays as in (\ref{JDecay})-(\ref{JDecayr}) with decay-rate  $D(r)< C r^{-\theta}$ for all $r>\delta_0$. Considering (\ref{EUBR}) and the decomposition (\ref{Pi123}), we find an upper-bound for $\mathcal E(u, B_R) $.  We start with subdomain $\Omega^1$. Note that on this subdomain we have $|x-y|>\delta_0$. From (\ref{JDecayr}), we conclude 
\begin{eqnarray*}
\iint_{ \Omega^1_R}  |x-y|  J(x-y)  dy dx &\le&   \int_{B_{R}}   \int_{|x-y|>R+\delta_0- |y|}   |y-x| J(y-x)  dx dy 
\\&=& \int_{B_{R}}  \sum_{k=1}^\infty \int_{k(R+\delta_0- |y|)<|x-y|<2k(R+\delta_0- |y|)}   |y-x| J(y-x)  dx dy 
\\&\le&   \int_{B_{R}}  (R+\delta_0 - |y|)^{1-\theta} dy  \left[ \sum_{k=1}^\infty k^{1-\theta} \right]\le C R^{n-1} \int_0^R  (R+\delta_0 - r)^{1-\theta} dr
\\&=&C\left[ \frac{\delta_0^{2-\theta}}{\theta-2}  - \frac{(R+\delta_0)^{2-\theta} }{\theta - 2} \right] R^{n-1} \le 
C\left[ \frac{\delta_0^{2-\theta}}{\theta-2} \right] R^{n-1} , 
\end{eqnarray*} 
when $\theta>2$ and  $C$ is a positive constant that is independent from $R$.  Similarly, for the subdomain $\Omega^2_R$,  we have $|x-y|>\delta_0$. From (\ref{JDecayr}),  we conclude 
\begin{eqnarray*}
\iint_{ \Omega^2_R}  |x-y|  J(x-y)  dy dx   & \le& C \left[ \int_{B_{R+\delta_0}\setminus B_R} dx \right]\left[ \sum_{k=1}^\infty \int_{k \delta_0<|z|<2k \delta_0} |z| J(z) dz\right]
\\&\le& C\left[ \sum_{k=1}^\infty k^{1-\theta} \right]  R^{n-1}   \le   C R^{n-1} . 
\end{eqnarray*} 
Note that due to the structure of the domain $\Omega^3_R$, a similar estimate as  (\ref{IntIER}) holds for the estimate on $\Omega^3_R$. This completes the proof.  

 \hfill $\Box$

\noindent {\it Proof of Theorem \ref{thmeneal}}.  The proof is similar to the one of Theorem \ref{thmene}.   We only  provide an upper-bound for the right-hand side of (\ref{EUBR}) with the domain decomposition  $\Omega_R=\cup_{i=1}^3 \Omega^i_R$ in (\ref{Pi123}). Consider a constant $\delta_0>0$. From $|u^t(x)- u^t(y)| \le C \min\{\delta_0, |x-y|\}$ and the  boundedness of $u$,  we have  
\begin{eqnarray}\nonumber
\mathcal E (u, B_R) &\le&  CR^{n-1} + C    \iint_{(\mathbb{R}^n\setminus B_R )\times B_R  } \left[\min\{\delta_0, |x-y|\}\right] J (x-y)  dy dx   
\\&\le & 
\label{EPhiuBR}  CR^{n-1} + C  \iint_{\Omega_R}  \left[\min\{\delta_0, |x-y|\}\right] J (x-y)  dy dx  . 
\end{eqnarray}
An upper-bound for the integral on $\Omega_R^3$ is given by (\ref{IntIER}).  We now compute the integral on $\Omega^1_R$ and will provide an upper-bound for the integral  
\begin{eqnarray*}
     \delta_0  \iint_{\Omega^1_R }|x-y|^{-n-\alpha}   dy dx
&=&    \delta_0   \int_{B_{R}   } \int_{ \mathbb{R}^n\setminus B_{R+\delta_0}(x) }   |z|^{-n-\alpha}   dz dx
\\&\le& \delta_0   \int_{B_{R}   } \int_{R+\delta_0-|x|}^{\infty}   r^{-1-\alpha}   dr dx
\\&\le&  \frac{\delta_0 }{\alpha}  \int_{B_{R}   }  (R +\delta_0- |x|)^{-\alpha} dx   
\\&\le&   \frac{ \delta_0  }{\alpha} R^{n-1} \int_{0}^{R}     (R +\delta_0- r)^{-\alpha} dr   . 
\end{eqnarray*}
Straightforward computations show that the latter integral is bounded by the following term, 
\begin{equation}\label{IntJ2R11}
\delta_0   \int_{B_{R}   } \int_{ \mathbb{R}^n\setminus B_{R+\delta_0} }   |x-y|^{-n-\alpha}   dy dx
\le  C   \left\{ \begin{array}{lcl}
\hfill  \delta_0  \log \left(\frac{R+\delta_0}{\delta_0} \right)  R^{n-1}  \ \ &\text{for}& \ \ \ \alpha=1,\\   
\hfill \frac{\delta_0}{\alpha(1-\alpha)} [(R+\delta_0)^{1-\alpha}- \delta^{1-\alpha}_0] R^{n-1}  \ \ &\text{for}&\ \ \alpha\neq 1 . 
\end{array}\right.
\end{equation}
Similar computations hold for subdomain $\Omega^2_R$. From (\ref{IntIER}), (\ref{EPhiuBR}) and (\ref{IntJ2R11})   we get the desired result.  

 \hfill $\Box$

\section{Pointwise Estimates and Monotonicity Formulas; Proofs of Theorem \ref{hamilton}-\ref{iden}}\label{secham}

In this section, we provide proofs for Theorem \ref{hamilton}-\ref{iden}.  In addition, we provide a monotonicity formula at the end for nonradial solutions.  The proofs of Theorem \ref{hamilton} and Theorem \ref{modica}  motivated by the ideas and methods provided in \cite{cabSire1}. 
\\
\\
\noindent\textbf{Proof of Theorem \ref{hamilton}:}  Suppose that $v$ is a solution of the extension problem. For any $x\in\mathbb R$, define 
\begin{equation}
w(x):=\frac{1}{2}  \int_0^\infty y^{a} \left[  (\partial_x v)^2 - (\partial_y v)^2 \right] dy.
\end{equation}
Differentiating with respect to $x$,  we get  
\begin{equation}\label{xw}
\partial_x w(x) =  \int_0^\infty y^{a} \left[  \partial_x v \partial_{xx} v - \partial_y v \partial_{xy} v \right] dy.
\end{equation}
From the first equation of (\ref{eg}),  we have 
$$ y^a \partial_{xx} v  + \partial_y \left(   y^a \partial_y v   \right)=0. $$
This and (\ref{xw}) yields 
\begin{equation}\label{xwi}
\partial_x w(x) =  \int_0^\infty \left[  -\partial_x v \partial_y \left(   y^a \partial_y v \right) -  y^a \partial_y v \partial_{xy} v \right] dy.
\end{equation}
From integration by parts,  we obtain the following
\begin{equation}
-\int_0^\infty \partial_x v  \partial_y \left(   y^a \partial_y v  \right) dy = \int_0^\infty y^a  \partial_{xy} v   \partial_y v   dy  +  \lim_{y\to 0} y^a    \partial_{x} v  \partial_{y} v . 
\end{equation}
From this and (\ref{xwi}),  we have 
 \begin{equation}
\partial_x w(x) =   \lim_{y\to 0} y^a    \partial_{x} v \partial_{y} v.
\end{equation}
From the boundary term in (\ref{main}) we get 
$$\partial_x w(x) = -d_\alpha \partial_{x} v [f(v(x,0))+\Delta v(x,0)] =d_\alpha \partial_{x} \left[ F(v(x,0)) -\frac{1}{2} (\partial_x v(x,0))^2\right] .
$$
Hence, 
$$\partial_x\left[    w(x)  - d_\alpha  F(v(x,0)) +  \frac{d_\alpha}{2} (\partial_x v(x,0))^2 \right]=0.$$
From  (\ref{asympve}),  we conclude for all $x\in\mathbb R$,  
\begin{equation}
  d_\alpha  F(v(x,0)) -d_\alpha  F(\tau)  =  \frac{d_\alpha}{2} (\partial_x v(x,0))^2 +w(x) . 
\end{equation}

               \hfill $ \Box$

\noindent\textbf{Sketch of Proof of Theorem \ref{modica}:} Define 
\begin{equation}
w_1(x,y):=\frac{1}{2}  \int_0^y t^{a} \left[  (\partial_x v(x,t))^2 - (\partial_y v(x,t))^2 \right] dt.
\end{equation}
Now define $ w_2(x,y):=F(v(x,0)) - F(\tau) - w_1(x,y)$. Then, it is straightforward to notice that 
\begin{equation}
\partial_y w_2(x,y) = - \frac{y^a}{2}  \left[  (\partial_x v(x,y))^2 - (\partial_y v(x,y))^2 \right], 
\end{equation}
and 
\begin{equation}
\partial_x w_2(x,y) = y^a \partial_x v(x,y) \partial_y v(x,y) +\partial_x v(x,0) \partial_{xx} v(x,0). 
\end{equation}
Now, define $w(x,y):= w_2(x,y) - \frac{1}{2} (\partial_x v(x,0))^2 $. We need to show that $w>0$  in $\mathbb R^2_+$. The function $w$ is bounded and its derivative satisfy $\partial_y w = \partial_y w_2$ and 
\begin{equation}
\partial_x w(x,y) = y^a \partial_x v(x,y) \partial_y v(x,y) . 
\end{equation}
From this and direct computations, for all $y>0$, we conclude  that 
\begin{equation}
\text{div} (y^a \nabla w)= -a y^{2a-1} ( \partial_x v )^2, 
\end{equation}
and 
\begin{equation}
\text{div} (y^{-a} \nabla w)= -a y^{-1} ( \partial_y v )^2. 
\end{equation}
Test of the proof is by contradiction assuming that $w$ does not attain its infimum in $(x,y)\in\mathbb R\times [0,\infty)$ and we omit it here. 

               \hfill $ \Box$

\noindent\textbf{Proof  of Theorem \ref{iden}:} 
Assume that $ v= v(r,y)$ for $r=|x|$ for $r\in\mathbb R^+$ and $y\in\mathbb R^+$. So, 
 \begin{eqnarray}\label{emainrad}
 \left\{ \begin{array}{lcl}
\hfill \partial_{rr} v +\frac{n-1}{r} \partial_r v + \partial_{yy} v+\frac{a}{y} \partial_y v&=& 0   \ \ \text{in}\ \  \mathbb R^+\times \mathbb R^+,\\   
\hfill -\lim_{y\to0}y^{a} \partial_{y} v&=&  d_{\alpha} [f(v(r,0))-\Delta v(r,0)]   \ \ \text{in}\ \ \mathbb R^+\times \{y=0\}.
\end{array}\right.
  \end{eqnarray}
For any $r\in \mathbb R^+$, define the following function,  
  \begin{equation}
  w(r):= \int_0^\infty   \frac{y^{a}}{2} \left[  (\partial_r v)^2 - (\partial_y v)^2 \right] dy .
  \end{equation}
Differentiating this with respect to $r$, we get 
 \begin{equation}\label{rwr}
  \partial_r w(r) =  \int_0^\infty   y^{a} \left[  \partial_r v  \partial_{rr} v - \partial_y v \partial_{ry} v \right] dy .
  \end{equation}
From the first equation in (\ref{emainrad}),  we have 
  \begin{equation}\label{}
 \partial_{rr} v = -\frac{n-1}{r} \partial_r v - \partial_{yy} v - \frac{a}{y} \partial_y v.
   \end{equation}
 Combining this  and  (\ref{rwr}),  the term $\partial_r w(r)$ can be rewritten as 
 \begin{equation}\label{wr}
 \partial_r w(r) =   -\frac{n-1}{r}   \int_0^\infty   y^{a} (\partial_r v)^2 dy -  \int_0^\infty   y^{a}  \partial_{yy} v \partial_{r} v dy 
 - a  \int_0^\infty   y^{a-1} \partial_{r} v \partial_{y} v dy -   \int_0^\infty   y^{a}  \partial_{y} v  \partial_{ry} v .
  \end{equation}
  Performing integration by parts yields  
  \begin{eqnarray*}
  \int_0^\infty   y^{a}  \partial_{y} v  \partial_{ry} v = -  \int_0^\infty  \partial_{y}\left(y^{a}  \partial_{y} v \right)  \partial_{r} v dy - \lim_{y\to 0} y^a  \partial_{r} v \partial_{y} v, 
  \end{eqnarray*}
that implies  
  \begin{equation*}
 \lim_{y\to 0} y^a  \partial_{r} v  \partial_{y} v  = - \int_0^\infty   y^{a}  \partial_{yy} v \partial_{r} v dy 
 - a  \int_0^\infty   y^{a-1} \partial_{r} v \partial_{y} v dy -  \int_0^\infty   y^{a}  \partial_{y} v  \partial_{ry} v .
   \end{equation*}
From this and (\ref{wr}),  for $r>0$, we get 
 \begin{equation}\label{partialrw} 
 \partial_{r} w(r) =  -\frac{n-1}{r} \int_0^\infty   y^{a} (\partial_r v)^2 dy + \lim_{y\to 0} y^a    \partial_{r} v \partial_{y} v .
 \end{equation}
From the second equation in (\ref{emainrad}), in we have  
  \begin{eqnarray*}
- \lim_{y\to 0} y^a   \partial_{y} v   \partial_{r} v &=&
 d_\alpha \partial_{r } \left[ f(v(r,0))+\Delta v(r,0)\right] \partial_{r} v  
\\ &=& d_\alpha  \partial_{r} \left( - F(v(r,0)) + \frac{1}{2} (\partial_r v(r,0))^2 \right)+\frac{n-1}{r} (\partial_r v(r,0))^2 . 
  \end{eqnarray*}
 From this and (\ref{partialrw}), we conclude 
$$ \partial_{r} \left( w(r) -  d_\alpha F(v(r,0)) + \frac{ d_\alpha}{2} (v_r(r,0))^2 \right)=   -\frac{n-1}{r}  \int_0^\infty   y^{a} (\partial_r v)^2 dy   -\frac{n-1}{r} (\partial_r v(r,0))^2  \le 0.$$ 
  This completes the proof. 
  

               \hfill $ \Box$

\begin{note} We fix the following notations throughout the paper; $B_R^+=\{X=( x,y)\in\mathbb R_+^{n+1}, | X|<R\}$ and $\partial^+ B_R^+=\partial B^+_R\cap \{y>0\}$.
\end{note}

\begin{prop} Let $ v$ be the extension function of solution $u$ of
\begin{equation}\label{ufu}
-\Delta u + c  (-\Delta)^{\frac{\alpha}{2}} u+ f(u)=0  \quad \text{in}\ \  \mathbb {R}^n. 
\end{equation}
For $R>1$ and $F'=f$, define  
\begin{equation}
I_{\gamma,\alpha}(R):= \frac{1}{R^{n-\gamma}} \left[  c\int_{B_R^+}  y^{a } |\nabla v |^2 d x dy + \int_{B_R}   |\nabla_x v(x,0) |^2 d x +2 \int_{B_R} F( v(x,0)) d x \right]. 
\end{equation}
Then, 
\begin{eqnarray*}
R^{n-\gamma+1} I'_{\gamma,\alpha}(R) &=&c \left(\frac{\gamma-\alpha}{2}\right) \int_{B_R^+}  y^{a } |\nabla v |^2 d x dy  + cR \int_{\partial^+{B_R^+}} y^{a} (\partial_{ \nu} v )^2 \\&&+ \left(\frac{\gamma-2}{2}\right) \int_{B_R}  |\nabla_x v(x,0) |^2 d x +  R \int_{\partial B_R} |\partial_r v(x,0)|^2  d \mathcal{H}^{n} +\gamma \int_{\partial B_R} F(v(x,0)) dx  . 
\end{eqnarray*}
\end{prop}  

\begin{proof} Let $c=1$. It is straightforward to show that for $ z\in\mathbb R^{n+1}_+$, 
\begin{equation}\label{div}
\div\left( y^{a}   z\cdot \nabla v \nabla v   -   \frac{1}{2}  y^{a}   |\nabla v|^2  z \right) + 
\frac{n-\alpha}{2} y^{a} |\nabla v|^2 =0.
\end{equation}
Integrate over $B_R^+$, and from the fact that $\partial^+ B_R^+$ we have $ z=R \nu$,  we conclude 
 \begin{eqnarray}\label{eq1}
\int_{B_R^+} \div\left( y^{a} \nabla v z\cdot \nabla v \right) &=& \lim_{y\to 0} \int_{B_R} y^{a} (-\partial_y v)  x\cdot \nabla_x v + \int_{\partial^+{B_R^+}} y^{a} (\partial_\nu v)^2   , 
\\ \label{eq2}
\int_{B_R^+} \div\left( y^{a}  z  |\nabla v|^2 \right) &=&  R \int_{\partial^+{B_R^+}} y^{a}     |\nabla v|^2 . 
\end{eqnarray}
From above, for (\ref{eg}), we have 
 \begin{equation}\label{id1}
 R \int_{\partial^+{B_R^+}} y^{a} (\partial_{ \nu} v )^2  + \int_{B_R} g(x)  x\cdot \nabla_{ x} v(x,0)  - \frac{R}{2} \int_{\partial^+{B_R^+}} y^{a}   |\nabla v|^2 +\frac{n-\alpha}{2} \int_{B_R^+}  y^{a} |\nabla v|^2 =0. 
 \end{equation}
For $g(x)=\Delta_x v(x,0)-f(v(x,0))$, we conclude the following Pohozaev identity 
\begin{eqnarray}\label{id2}
 R \int_{\partial^+{B_R^+}} y^{a} (\partial_{ \nu} v )^2  +\frac{n-2}{2} \int_{B_R} |\nabla_x v(x,0)|^2 -\frac{R}{2} \int_{\partial B_R} |\nabla_x v(x,0)|^2 + R \int_{\partial B_R} |\partial_r v(x,0)|^2 \\- R \int_{\partial B_R} F(v(x,0)) +n \int_{B_R} F(v(x,0))
 - \frac{R}{2} \int_{\partial^+{B_R^+}} y^{a}   |\nabla v|^2 +\frac{n-\alpha}{2} \int_{B_R^+}  y^{a} |\nabla v|^2 =0  . 
 \end{eqnarray}
Differentiating $I_{\gamma,\alpha}(R)$, we get
\begin{eqnarray*}
I_{\gamma,\alpha}'(R)&=&-\left(\frac{n-\gamma}{2}\right)R^{-n+\gamma-1} \int_{B_R^+}  y^{a } |\nabla v |^2 +  \frac{R^{-n+\gamma}}{2}\int_{\partial^+ B_R^+}  y^{a } |\nabla v |^2 \\&&-\left(\frac{n-\gamma}{2}\right)R^{-n+\gamma-1} \int_{B_R}  |\nabla_x v(x,0) |^2 + \frac{R^{-n+\gamma}}{2}  \int_{\partial B_R}  |\nabla_x v(x,0) |^2 
\\&&+ (-n+\gamma) R^{-n+\gamma-1} \int_{B_R} F(v(x,0)) + R^{-n+\gamma}  \int_{\partial B_R} F(v(x,0)) . 
\end{eqnarray*}
Combining the above two equalities, completes the proof. 

\end{proof}

As a direct consequence of the above technical computations,  we conclude that if the following inequality holds for any $\gamma,c,\alpha$
 \begin{equation}\label{ingam}
 2\gamma \int_{B_R} F(v(x,0)) dx \ge c (\alpha-\gamma)  \int_{B_R^+}  y^{a } |\nabla v |^2 d x dy + (2-\gamma) \int_{B_R}   |\nabla_x v(x,0) |^2 d x. 
 \end{equation} 
Then, $I_{\gamma,\alpha}(R)$ is a nondecreasing function of $R$ when $F \ge 0$.  

\begin{thm}
If $\gamma=2$, then $I_{\gamma,\alpha}(R)$ is  a nondecreasing function of $R$ when $F \ge 0$. 
\end{thm}
As a direct consequence of the above monotonicity formula, one can conclude that $v(x,y)$ is constant for $n\ge 2$ and 
$$
 c\int_{\mathbb R^{n+1}_+}  y^{a } |\nabla v |^2 d x dy + \int_{\mathbb R^{n}}   |\nabla_x v(x,0) |^2 d x +2 \int_{\mathbb R^{n}} F( v(x,0)) d x <\infty .
$$
Note that for the case of local equations that is $c=0$,  Modica's estimate in \cite{mod} implies that the above inequality holds for $\gamma=1$. This implies that for this case the monotonicity formula holds for $\gamma=1$.  This raises the natural question that  if (\ref{ingam}) holds when  $c>0$,  $0<\alpha<2$ and $\gamma<2$. This remains as an open problem.   

\section{Summation of Nonlocal Operators}\label{secsum}
 In this section, we consider the sum of singular jump kernels of the form  $J=J_1+J_2$ where both $J_1$ and $J_2$ are nonnegative measurable symmetric  even jump kernels.  The nonlocal equation, without the diffusion component, associated with this kernel is 
 \begin{equation} \label{Tmain}
  \lim_{\epsilon\to 0} \int_{\{y\in \mathbb R^n, |x-y|>\epsilon\} } [u(y) - u(x)] \left(J_1 (x,y) +J_2(x,y)\right)dy +  f(u)=0  \quad  \text{in} \ \  \RR^n.  
  \end{equation}   
 Inspired by the fractional Laplacian operator, naturally, we consider 
\begin{equation}\label{Jc1c2}
 J (x,z) =  \frac{c_1(x-z)}{|x-z|^{n+\alpha_1}} + \frac{c_2(x-z)}{|x-z|^{n+\alpha_2}}, 
 \end{equation} 
where $0<\alpha_1,\alpha_2<2$ and $c_{1}$ and $ c_{2}$ are bounded between two positive constants $0<\lambda \le \Lambda$.  When $c_1$ and $c_2$ are constant, then the above operator is the sum of two fractional Laplacian operator $\Delta^{\frac{\alpha_1}{2}} + \Delta^{\frac{\alpha_2}{2}} $. The sum of fractional powers of Laplacian operators have been studied in the literature. The authors in \cite{svon,fot,cks} and references therein studied such operators and established Harnack inequalities and heat kernel estimates. In addition, Silvestre in \cite{si} studied  H\"{o}lder estimates and regularity properties, and  Cabr\'{e} and Serra in \cite{cs} provided  symmetry results, among other interesting results, via proving and applying the extension problem for such operators.   The associated energy functional for solutions of (\ref{Tmain}) is given by $\mathcal E(u,\Omega)$, in (\ref{energy}) for  $\Omega\subset \mathbb R^n$, when the functional   $\mathcal E_J^{\text{Sob}}$ is  
\begin{equation}\label{Esob1}
\mathcal E_J^{\text{Sob}}(u,\Omega):=  \frac{1}{2} \iint_{\mathbb R^{n}\times \mathbb R^{n} \setminus  \mathcal C\Omega \times   \mathcal C\Omega}  | u(x) -u(y) |^2\left( J_1 (x-y) +J_2 (x-y) \right) dy dx, 
\end{equation}
and $\mathcal E^{\text{Pot}} $ is given by (\ref{Esob}).  For this energy functional, one can mimic the proofs in previous sections to establish the following estimate. 
\begin{thm}\label{thmeneal2}
Suppose that $ u$ is a bounded monotone solution of (\ref{Tmain}) with  $ F(\pm 1)=0$.  Assume also that the kernel $J$ satisfies  (\ref{Jc1c2}).   Then,  the following energy estimates hold for $R>1$. 
\begin{enumerate}
\item[(i)] If $0<\min\{\alpha_1, \alpha_2\}<1$, then $\mathcal E(u,B_R)  \le  C R^{n-\min\{\alpha_1,\alpha_2\}}$,
\item[(ii)] If  $\min\{\alpha_1, \alpha_2\}=1$, then $\mathcal E(u,B_R) \le  C R^{n-1}\log R$,
\item[(iii)] If $\min\{\alpha_1, \alpha_2\}>1$, then $\mathcal E(u,B_R)  \le  C R^{n-1}$,
\end{enumerate} 
where the positive constant $C$ is independent from $R$ but may depend on $\alpha_1,\alpha_2,n,\lambda,\Lambda$. 
\end{thm}
Applying Pohozaev-type arguments, one can see that the following monotonicity formula holds for the extension function. 
\begin{prop}\label{propmon}
Let $ v$ be the extension function of solution of (\ref{Tmain}) when the kernel $J$ is given by (\ref{Jc1c2}) when $c_1$ and $c_2$ are constant.  For $R>1$ and $a_i=1-2\alpha_i$,  define  
\begin{equation}
I_{\gamma,\alpha_1,\alpha_2}(R):= \frac{1}{R^{n-\gamma}} \left[  \int_{B_R^+} ( y^{a_1 }+  y^{a_2})  |\nabla v |^2 d x dy +2 \int_{B_R} F( v(x,0)) d x \right]. 
\end{equation}
Then, 
\begin{eqnarray*}
R^{n-\gamma+1} I'_{\gamma,\alpha_1,\alpha_2}(R) &=& \left(\frac{\gamma-\alpha_1}{2}\right) \int_{B_R^+}  y^{a_1 } |\nabla v |^2 d x dy  + R \int_{\partial^+{B_R^+}} y^{a_1} (\partial_{ \nu} v )^2 \\&&+ \left(\frac{\gamma-\alpha_2}{2}\right) \int_{B_R^+}  y^{a_2 } |\nabla v |^2 d x dy  + R \int_{\partial^+{B_R^+}} y^{a_2} (\partial_{ \nu} v )^2+\gamma \int_{\partial B_R} F(v(x,0)) dx .
\end{eqnarray*}
\end{prop}  

\begin{cor}
If $\gamma=\max\{\alpha_1,\alpha_2\}$, then $I_{\gamma,\alpha_1,\alpha_2}(R)$ is  a nondecreasing function of $R$ when $F \ge 0$. 
\end{cor}
Since the Poincar\'{e} type inequality in Theorem \ref{ThmPoin} holds for a general kernel, we can establish De Giorgi type results in two dimensions when $J=J_1+J_2$ where for $i=1,2$, 
\begin{equation}\label{JJ1}
\frac{\lambda}{|x-y|^{n+\beta_i}} \mathds{1}_{\{|x-y|\le \delta_1\}} \le  J_i  (x-y) \le  \frac{\Lambda}{|x-y|^{n+ \alpha_i}} \mathds{1}_{\{|x-y|\le \delta_0\}}. 
 \end{equation} 
\begin{thm}\label{Thmain3} Let $u$  be a bounded stable solution of  (\ref{Tmain}) in two dimensions when the jump kernel $J=J_1+J_2$ is truncated and satisfies  (\ref{JJ1}).  Then,  $u$ must be a one-dimensional function.   
  \end{thm}
The proofs of above results are eliminated due to the similarity to the ones in Section \ref{secmain}. We end this section with pointing out that our main results can be easily generalized  to the case when the jumping measure is of the form $J=\sum_{i=1}^m J_i$ for any $m\in\mathbb N$. 
\\
\\
\noindent {\it Acknowledgment}.   The  author would like to thank Professor Yannick Sire for discussions and comments on this topic.

\end{document}